\documentclass[reqno]{amsart}

\usepackage{amsmath,mathrsfs,amssymb,multirow,setspace,enumerate}
\usepackage{graphicx}
\usepackage[colorlinks=true,allcolors=black]{hyperref}
\usepackage{relsize}

\theoremstyle{plain}

\newtheorem{theorem}{Theorem}[section]
\newtheorem{lemma}[theorem]{Lemma}
\newtheorem{proposition}[theorem]{Proposition}
\newtheorem{corollary}[theorem]{Corollary}

\theoremstyle{definition}
\newtheorem{notation}[theorem]{Notations}

\theoremstyle{remark}
\newtheorem{remark}[theorem]{Remark}

\input xy
\xyoption{all}

\begin{document}
\title[On the Enhanced Power Graph of  a Finite Group]{On the Enhanced Power Graph of a Finite Group}

\author[R. P. Panda]{Ramesh Prasad Panda}

\author[S. Dalal]{Sandeep Dalal}

\author[J. Kumar]{Jitender Kumar}

\begin{abstract}
The enhanced power graph $\mathcal{P}_e(G)$ of a group $G$ is a graph with vertex set $G$ and two vertices are adjacent if they belong to the same cyclic subgroup. In this paper, we consider the minimum degree, independence number and matching number of enhanced power graphs of finite groups. We first study these graph invariants for $\mathcal{P}_e(G)$ when $G$ is any finite group, and then determine them when $G$ is a finite abelian $p$-group, $U_{6n} = \langle a, b : a^{2n} = b^3 = e, ba =ab^{-1} \rangle$, the dihedral group $D_{2n}$, or the semidihedral group  $SD_{8n}$.
If $G$ is any of these groups, we prove that $\mathcal{P}_e(G)$ is perfect and then obtain its strong metric dimension.
Additionally, we give an expression for the independence number of $\mathcal{P}_e(G)$ for any finite abelian group $G$.
 These results along with certain known equalities yield the edge connectivity, vertex covering number and edge covering number of enhanced power graphs of the  respective groups as well.
\end{abstract}



\maketitle

 \section{Introduction}
 The notion of power graph of a group was introduced as a directed graph by Kelarev and Quinn in \cite{a.kelarev2000groups} and was later extended to semigroups in \cite{a.kelarev2001powermatrices,kelarevDirectedSemigr}. The undirected power graph or simply the power graph of a semigroup, which is the underlying undirected graph of the above, was first studied in \cite{GhoshSensemigroups}. The \emph{power graph} $\mathcal{P}(G)$ of a group $G$ is an undirected and simple graph whose vertices are the elements of $G$ and two vertices are adjacent if one of them is a power of the other.

 For any group $G$, it was shown in \cite{GhoshSensemigroups} that $ \mathcal{P}(G) $ is connected if and only if $G$ is periodic. Moreover, when $G$ is finite, $ \mathcal{P}(G) $ is complete if and only if $ G $ is a cyclic group of order one or prime power. Cameron and Ghosh \cite{Ghosh} proved that two finite abelian groups with isomorphic power graphs are isomorphic. More generally, Cameron \cite{Cameron} proved that if two finite groups have isomorphic power graphs, then their directed power graphs are also isomorphic. Many other interesting results on power graphs have been obtained in literature, see \cite{curtin2014edge,a.strongmetricdim2018,MR3200118,panda2019combinatorial} and the references therein. Additionally, we can refer to the survey paper by Abawajy et al. \cite{kelarevPGsurvey} for various results and open problems on directed/undirected power graphs.

 In the remainder of the paper, graphs considered are undirected and simple. For any group $G$, the {\em enhanced power graph} $\mathcal{P}_e(G)$ of $G$ is a graph whose vertices are the elements of $G$ and two vertices are adjacent if they belong to the same cyclic subgroup. The enhanced power graph of a group was introduced by Aalipour et al. \cite{aalipour2017}. It can be observed that for any group $G$, the graph $\mathcal{P}(G)$ is a spanning subgraph of $\mathcal{P}_e(G)$. The following result characterizes the equality of these graphs.

 \begin{lemma}[{\cite{aalipour2017}}]\label{pg.epg}
 	For any finite group $G$, $\mathcal{P}_e(G) = \mathcal{P}(G)$ if and only if every cyclic subgroup of $ G $ has prime power order.
 \end{lemma}

  In  \cite{aalipour2017}, the clique number of enhanced power graphs were obtained in terms of orders of elements of the corresponding groups.
 Hamzeh et al. \cite{MR3567537} derived the automorphim group of enhanced power graphs of finite groups.
In \cite{bera}, Bera and Bhuniya characterized the abelian groups and the nonabelian $p$-groups having dominatable enhanced power graphs.
Ma and She \cite{ma2019metric} investigated the metric dimension of enhanced power graphs of finite groups.
Zahirovi\'{c} \cite{samir} et al. supplied a characterization of finite nilpotent groups whose enhanced power graphs are perfect. In addition to enhanced power graph, other variants of power graphs have been considered recently, see \cite{MR3963087,MR3609333}.

The \emph{cyclic graph} of a semigroup $S$ is a graph whose vertex set is $S$ and two vertices $u$ and $v$ are adjacent if $ \langle u,v \rangle = \langle w \rangle $ for some $w \in S$. It can be observed that the notion of  enhanced power graph and cyclic graph coincide when $S$ is a group. Cyclic graphs have been studied for groups and semigroups in \cite{MR3200113,dalal2019,ma2013cyclic}.

 In the next section, we state necessary fundamental notions and recall some existing results.
 Then in Section \ref{main}, we investigate various structural properties of the enhanced power graph of a finite group $G$. We determine the minimum degree and independence number of $\mathcal{P}_e(G)$. Then we obtain the matching number of $\mathcal{P}_e(G)$ when $G$ is of odd order and its bounds when $G$ is of even order.
In Subsection \ref{pgroup}, we obtain the matching number of $\mathcal{P}_e(G)$ when $G$ is any finite $p$-group ($p$ is prime), and the independence number and strong metric dimension of $\mathcal{P}_e(G)$ when $G$ is a finite abelian $p$-group.
Accordingly, if $G$ is any finite abelian group, we give an expression of the independence number of $\mathcal{P}_e(G)$.
 In Sections 3.2 to 3.4, we first prove that  $\mathcal{P}_e(G)$ is perfect and then compute the aforementioned graph invariants of $\mathcal{P}_e(G)$ when $G$ is $U_{6n}$, $D_{2n}$ or $SD_{8n}$, respectively. In view of Lemma \ref{independence.vertex.cover} and Lemma \ref{edgecon.mindeg}, these results determines the edge connectivity, vertex covering number and edge covering number  of $\mathcal{P}_e(G)$ as well.

\section{Preliminaries}
\label{prelim}

	Consider a group $G$ with the identity element $e$. If $H$ is any subgroup of $G$, we simply write $H \leq G$. The \emph{exponent} of $G$ is the least common multiple of orders of its elements. An \emph{involution} (if exists) in $G$ is an element of order $2$. Given an equivalence relation $\rho$ on a set $A$, we refer to an equivalence class under $\rho$ in $A$  simply as a $\rho$-class.
	
 For any $x \in G$, $\langle x \rangle$ is the cyclic subgroup of $G$ generated by $x$. For any $x, y \in G$, we write $x \approx y$ if $\langle x \rangle = \langle y \rangle$. Observe that $\approx$ is an equivalence relation on $G$. We denote by $[x]$ the $\approx$-class containing $x$. Note that $[x]$ is precisely the set of generators of $\langle x \rangle$. For any $n \in \mathbb{N}$, consider $\mathbb{Z}_n$, the group of integers modulo $n$. We denote by $\mathbb{Z}^m_{n}$ the group obtained by taking direct product of $m$ copies of $\mathbb{Z}_{n}$.

 \begin{lemma}[{\cite[Theorem 5.2.4]{robinson1996groups}}]\label{nilpotent}
	A finite group $G$ is nilpotent if and only if $G$ is isomorphic to a direct product of its Sylow subgroups.
\end{lemma}


For integers $a$ and $b$, we denote their greatest common divisor by $(a,b)$. The following result about the Euler's totient function $\phi$ is well known (for instance, see \cite{burton2006elementary}).

\begin{lemma}\label{phieven}
	For any integer $n \geq 3$, $\phi(n)$ is even.
\end{lemma}

Consider a graph $ \Gamma $. The vertex set and edge set of $\Gamma$ are denoted by $V(\Gamma)$ and $E(\Gamma)$, respectively.
If two vertices $u$ and $v$ are adjacent, we write $u \sim v$, otherwise, $u \nsim v$.
If $a$ is any vertex or edge of $\Gamma$, we denote by $\Gamma - a$ the subgraph of $\Gamma$ obtained by deleting $a$.
A \emph{block} $B$ in $ \Gamma $ is a maximal connected subgraph such that $B - v$ is connected for all $v \in V(B)$.
A \emph{clique} in a graph $\Gamma$ is a set of vertices that induces a complete subgraph. The maximum cardinality of a clique in $\Gamma$ is called the \emph{clique number} of $\Gamma$ and it is denoted by $\omega({\Gamma})$.
The \emph{chromatic number} $\chi(\Gamma)$ of a graph $\Gamma$ is the smallest positive integer $k$ such that the vertices of $\Gamma$ can be colored in $k$ colors so that no two adjacent vertices share the same color.
The graph $\Gamma$ is \emph{perfect} if $\omega(\Gamma') = \chi(\Gamma')$ for every induced subgraph $\Gamma'$ of $\Gamma$.
Recall that the {\em complement} $\overline{\Gamma}$ of $\Gamma$ is a graph with same vertex set as $\Gamma$ and distinct vertices $u, v$ are adjacent in $\overline{\Gamma}$ if they are not adjacent in $\Gamma$. A subgraph $\Gamma'$ of $\Gamma$ is called \emph{hole} if $\Gamma'$ is a  cycle as an induced subgraph, and $\Gamma'$ is called an \emph{antihole} of   $\Gamma$ if   $\overline{\Gamma'}$ is a hole in $\overline{\Gamma}$.

\begin{theorem}[\cite{a.strongperfecttheorem}]\label{strongperfecttheorem}
	A finite graph $\Gamma$ is perfect if and only if it does not contain hole or antihole of odd length at least $5$.
\end{theorem}

It is known that $e$ is adjacent to all other vertices of $ \mathcal{P}_e(G) $ for any finite group $G$. A consequence of this is the following.

\begin{remark}\label{rem.e}
	For any finite group $G$, $e$ does not belong to the vertex set of any hole of length greater than $3$, or any antihole of $ \mathcal{P}_e(G)$.	
\end{remark}

We use Ramark \ref{rem.e} without explicitly referring to it.

The minimum degree of $ \Gamma $, denoted by $ \delta(\Gamma) $, is the minimum of degrees of vertices of $ \Gamma $.  An {\em independent set} of $\Gamma$ is a set of vertices none two of which are adjacent in $ \Gamma $. The {\em independence number} $ \alpha(\Gamma) $ of $\Gamma$ is the largest cardinality of an independent set of $\Gamma$.
 A \emph{matching} of $\Gamma$ is a set of edges such that no two of them are incident to the same vertex. The \emph{matching number} $\alpha'(\Gamma)$ of $\Gamma$ is the maximum  cardinality of a matching.
  A \emph{vertex cover} of $ \Gamma $ is a set of vertices that contains at least one endpoint of every edge in $ \Gamma $. The \emph{vertex covering number} $\beta(\Gamma)$ of $ \Gamma $ is the minimum cardinality of a vertex cover.
  An \emph{edge cover} of $\Gamma$ is a set $E$ of edges such that every vertex of $\Gamma$ is incident to some edge in $E$. The \emph{edge covering number} $\beta'(\Gamma)$ of $\Gamma$ is the minimum cardinality of an edge cover in $\Gamma$. We have the following equalities involving the these four invariants.

  \begin{lemma}[{\cite{WestGT}}]\label{independence.vertex.cover}
  	Consider a graph $\Gamma$.
  	\begin{enumerate}[\rm(i)]
  		\item $\alpha(\Gamma) + \beta(\Gamma) = |V(\Gamma)|$.
  		\item  If $ \Gamma $ has no isolated vertices, $\alpha'(\Gamma) + \beta'(\Gamma) = |V(\Gamma)|$.	
  	\end{enumerate}	
  \end{lemma}

Let $G$ be a group. For any $A \subseteq G$, we denote by $\mathcal{P}_e(A)$, the subgraph of $ \mathcal{P}_e(G) $ induced by $A$. The \emph{neighbourhood} $ N(x) $ of a vertex $x$ is the set all vertices adjacent to $x$ in $ \mathcal{P}_e(G) $. Additionally, we denote $N[x] = N(x) \cup \{x\}$.

For any $x \in G$ with $o(x) \geq 3$, the set $[ x ]$ of vertices is a clique in $\mathcal{P}_e(G)$. Since  $|[ x ]| = \phi(o(x))$, we have a matching, denoted by $M_x$, of order $\dfrac{\phi(o(x))}{2}$ consisting of edges with ends in $[ x ]$.

For  vertices $u$ and $v$ in a graph $\Gamma$, we say that $z$ \emph{strongly resolves} $u$ and $v$ if there exists a shortest path from $z$ to $u$ containing $v$, or a shortest path from $z$ to $v$ containing $u$. A subset $U$ of $V(\Gamma)$ is a \emph{strong resolving set} of $\Gamma$ if every pair of vertices of $\Gamma$ is strongly resolved by some vertex of $U$. The least cardinality of a strong resolving set of $\Gamma$ is called the \emph{strong metric dimension} of $\Gamma$ and is denoted by $\operatorname{sdim}(\Gamma)$. For  vertices $u$ and $v$ in a graph $\Gamma$, we write $u\equiv v$ if $N[u] = N[v]$. Notice that that $\equiv$ is an equivalence relation on $V(\Gamma)$.
We denote by $\widehat{v}$ the $\equiv$-class containing a vertex $v$ of $\Gamma$.
 Consider a graph $\widehat{\Gamma}$ whose vertex set is the set of all $\equiv$-classes, and vertices $\widehat{u}$ and  $\widehat{v}$ are adjacent if $u$ and $v$ are adjacent in $\Gamma$. This graph is well-defined because in $\Gamma$, $w \sim v$ for all $w \in \widehat{u}$ if and only if $u \sim v$.  We observe that $\widehat{\Gamma}$ is isomorphic to the subgraph $\mathcal{R}_{\Gamma}$ of $\Gamma$ induced by a set of vertices consisting of exactly one element from each $\equiv$-class. Subsequently, we have the following result of \cite{a.strongmetricdim2018} with $\omega(\mathcal{R}_{\Gamma})$ replaced by $\omega(\widehat{\Gamma})$.

\begin{theorem}[{\cite[Theorem 2.2]{a.strongmetricdim2018}}]\label{strong-metric-dim}
For any graph $\Gamma$ with diameter $2$,  $\operatorname{sdim}(\Gamma) = |V(\Gamma)| - \omega(\widehat{\Gamma})$.
\end{theorem}

When $\Gamma = \mathcal{P}_e(G)$ for some group $G$, we denote $\widehat{\Gamma}$ by $\widehat{\mathcal{P}}_e(G)$. Also, we denote its vertex set by $\widehat{G}$.

We often use the following fundamental property of enhanced power graphs without referring to it  explicitly.

\begin{lemma}[\cite{bera,ma2013cyclic}]
	For any finite group $G$, the graph $\mathcal{P}_e(G)$ is complete if and only if $G$ is cyclic.
\end{lemma}

\begin{lemma}[{\cite[Proposition 4]{Cameron}}]\label{cameron}
Let $G$ be a finite group, and $S$ be the set of vertices of $\mathcal{P}(G)$ that are adjacent to all other vertices. If $|S|>1$, then one of the following occurs.
\begin{enumerate}[\rm(i)]
\item $G$ is cyclic of prime power order,and $S=G$,
\item $G$ is cyclic of non-prime-power order $n$ and $S$ consists of the identity and the generators of $G$,
\item $G$ is generalized quaternion and $S$ contains the identity and the unique involution.
\end{enumerate}
\end{lemma}

We end this section with the following property of power graphs of finite $p$-groups.


\begin{lemma}[{\cite[Proposition 3.2]{power2017conn}}]\label{Componentp}
	For any prime $p$ and finite $p$-group $G$, each component of $\mathcal{P}(G  {\setminus} \{e\})$ has exactly $p-1$ elements of order $p$.
\end{lemma}

\section{Main results}
\label{main}

This section contains the results obtained in this paper. We first consider the minimum degree, independence number and matching number of enhanced power graphs of arbitrary finite groups. Then we determine these graph invariants and strong metric dimension for abelian $p$-groups, $U_{6n}$, $D_{2n}$ and $SD_{8n}$ in respective subsections. Moreover, we prove the enhanced power graph of each of $U_{6n}$, $D_{2n}$ and $SD_{8n}$ is perfect.

If the diameter of any graph is at most $ 2 $, then its edge connectivity and minimum degree are equal (cf. \cite{plesnik1975critical}). A consequence of this is the following result.

\begin{lemma}\label{edgecon.mindeg}
	For any finite group $G$, the edge connectivity and minimum degree of $\mathcal{P}_e(G)$ coincide.
\end{lemma}

We begin with minimum degree of enhanced power graphs of finite groups.
In light of Lemma \ref{edgecon.mindeg}, the following determines the edge connectivity of $\mathcal{P}_e(G)$ as well.

\begin{theorem}\label{mindeg}
	For any finite group $G$, the minimum degree $\delta(\mathcal{P}_e(G))=m-1$, where $m$ is the order of a smallest maximal cyclic subgroup of $G$.
\end{theorem}

\begin{proof}
	Let $x \in G$. Then $x$ belongs some maximal cyclic subgroup, say $ \langle y \rangle $, of $G$. Since $\langle y \rangle$ induces a clique in $\mathcal{P}_e(G)$, we have $ \deg(x) \geq o(y) -1$. Moreover, as $\langle y \rangle$ is a maximal cyclic subgroup of $G$, $N(y) = \langle y \rangle \setminus \{y\}$. This implies $ \deg(y) =o(y)-1 $, so that $\deg(x) \geq \deg(y)$. Now let $M$ be a maximal cyclic subgroup of $G$ of least order. Then $\deg(x) \geq |M|-1$ for any $x \in G$, and that $\deg(z) = |M|-1$ for any $z$ generating $M$. Accordingly, the proof follows.
\end{proof}

Now we study the independence number of enhanced power graphs.

\begin{theorem}\label{indep}
	For any finite group $G$, the independence number of $\mathcal{P}_e(G)$ coincides with the number of maximal cyclic subgroups of $G$.
		Furthermore, if $G$ is nilpotent and the prime factors of $|G|$ are $p_1,p_2,\dots, p_r$, then the independence number $$\alpha(\mathcal{P}_e(G)) = m_1m_2 \cdots m_r,$$
		 where $m_i$ is the number of maximal subgroups of a Sylow-$p_i$ subgroup.
\end{theorem}

\begin{proof}
	Let $\mu(G)$ denote the number of maximal cyclic subgroups of $G$.
	If $x$ and $y$ are two elements generating two different maximal cyclic subgroups of $G$, then they are non-adjacent in $\mathcal{P}_e(G)$. As a result, $\alpha(\mathcal{P}_e(G)) \geq \mu(G)$. Now consider an independent set $S$ in $\mathcal{P}_e(G)$.
	Recall that any group can be written as union of its maximal cyclic subgroups. Since a maximal cyclic subgroup induces a clique in $\mathcal{P}_e(G)$, no two members of an independent set in $\mathcal{P}_e(G)$ belong to the same maximal cyclic subgroup. As a result, $\alpha(\mathcal{P}_e(G))  \leq  \mu(G)$.   Thus we conclude that $\alpha(\mathcal{P}_e(G)) = \mu(G)$.
	
	Next let $G$ be nilpotent and $P_i$ be a Sylow-$p_i$ subgroup of $G$ for $1 \leq i \leq r$. By Lemma \ref{nilpotent}, we gave $G = P_1P_2 \cdots P_r$.
	Then $H$ is a maximal cyclic subgroup of $ G $ if and only if $H =  H_1 H_2 \cdots H_r $, where $ H_i $ is a maximal cyclic subgroup of $P_i$ for $ 1 \leq  i \leq r $ (see \cite{power2018cutset}).
	
	Let $ H_i $, $ H_i'$ are maximal cyclic subgroups of $P_i$ for $ 1 \leq  i \leq r $. If $ H_k \neq H_k'$ for any $ 1 \leq  k \leq r $, then $H_1 H_2 \cdots H_r \neq H_1' H_2' \cdots H_r'$. This is because the generators of $ H_k$ belong to $H_1 H_2 \cdots H_r$, but not to $H_1' H_2' \cdots H_r'$. Therefore, if the number of maximal subgroups of $P_i$ is  $m_i$, then $\alpha(\mathcal{P}_e(G)) = m_1m_2\cdots m_r$.
\end{proof}

In view of Lemma \ref{independence.vertex.cover}(i), we have the following consequence of Theorem \ref{indep}.

\begin{corollary}
	For any finite group $G$, the vertex covering number $\beta(\mathcal{P}_e(G)) = |G|-\mu(G)$, where $\mu(G)$ is the number of maximal cyclic subgroups of $G$.
\end{corollary}

We next compute values and bounds of matching number of enhanced power graphs of finite groups.

\begin{theorem}\label{matching}
	Let $G$ be a finite group. If $G$ is of odd order, then the matching number $\alpha'(\mathcal{P}_e(G)) = \dfrac{|G|-1}{2}$.
	If $G$ is of even order, then $$\dfrac{|G|-(t-1)}{2} \leq \alpha'(\mathcal{P}_e(G)) \leq \dfrac{|G|}{2},$$ where $t$ is the number of involutions in $G$.
\end{theorem}

\begin{proof}
First let $G$ be of odd order. Observe that for distinct $x_1, x_2 \in G {\setminus} \{e\}$, either $[ x_1 ] = [ x_2 ]$ or $[ x_1 ] \cap [ x_2 ] = \emptyset$. Accordingly, either $M_{x_1} = M_{x_2}$ or $ \displaystyle M_{x_1} \cap M_{x_2} = \emptyset$. Hence $M := \cup_{x \in G{\setminus} \{e\}} M_{x}$ is a matching of order $\dfrac{|G|-1}{2}$ in $ \mathcal{P}_e(G) $. On the other hand, the order of a largest matching in a graph of order $n$ is $\left \lfloor \dfrac{n}{2} \right \rfloor$. Hence we get $\alpha'(\mathcal{P}_e(G)) = \dfrac{|G|-1}{2}$.

Now suppose $G$ is of even order. Then it has at least one involution, say $y$. We denote the edge with ends $e$ and $y$ by $\epsilon$. Then $M = \{\epsilon\} \bigcup \cup_{x \in G, o(x) \geq 3} M_{x}$ is a matching of order $\dfrac{|G|-(t-1)}{2}$ in $ \mathcal{P}_e(G) $, where $t$ is the number of involutions in $G$. Additionally, as $\alpha'(\mathcal{P}_e(G)) \leq \dfrac{|G|}{2}$ holds trivially, we get the desired inequality when $G$ is of even order.	
\end{proof}

Considering Lemma \ref{independence.vertex.cover}(ii), we have the following corollary of Theorem \ref{matching}.

\begin{corollary}\label{edge.cover}
	Let $G$ be a finite group. If $G$ is of odd order, then the edge covering number $\beta'(\mathcal{P}_e(G)) = \dfrac{|G|+1}{2}$.
	If $G$ is of even order, then $$\dfrac{|G|}{2} \leq \beta'(\mathcal{P}_e(G)) \leq \dfrac{|G|+(t-1)}{2},$$ where $t$ is the number of involutions in $G$.
\end{corollary}

From the two preceeding results, for any finite group $G$ with an unique involution, $\alpha'(\mathcal{P}_e(G)) = \beta'(\mathcal{P}_e(G)) = \dfrac{|G|}{2}$.
Recall that for any prime $p$, a finite $p$-group $G$ has exactly one subgroup of order $p$ if and only if $G$ is cyclic, or $p=2$ and $G$ is  generalized quaternion (see \cite{robinson1996groups}). This facts along with Lemma \ref{nilpotent} yield the following corollary.

\begin{corollary}
If $G$ is a nilpotent group with a cyclic or generalized quaternion Sylow-$2$ subgroup, then
 $\alpha'(\mathcal{P}_e(G)) = \beta'(\mathcal{P}_e(G)) = \dfrac{|G|}{2}$.
\end{corollary}

Now we investigate various structural properties of enhanced power graphs of the groups under consideration.

\bigskip
\subsection{Finite $p$-group.}\hfill
\label{pgroup}
\smallskip

In this subsection, $p$ denotes a prime number.
Following Lemma \ref{pg.epg}, power graphs and enhanced power graphs coincide for finite $p$-groups. We therefore consider the former in the following.
 Let $G$ be a finite $p$-group. Then $\mathcal{P}(G)$ is perfect, since more generally, the power graph of any finite group is perfect (see \cite{dooser}). It is known that a finite abelian group $G$ is isomorphic to a unique direct product of cyclic groups
 of prime power order. In this product, let $\tau(G)$ be the order of the smallest cyclic group. Then the following is a consequence of Theorem \ref{mindeg}.

\begin{theorem}[{\cite{power2017mindeg}}]
	For any finite abelian $p$-group $G$, the minimum degree of $\mathcal{P}(G)$ is $\tau(G)-1$.
\end{theorem}

Before proceeding further, we are required to fix some notations.

\begin{notation}\label{not1}
  Consider a prime $p$ and positive integers $\alpha_1 > \dots > \alpha_s $ and $ m_1, \dots, m_s$. For any $1 \leq j \leq s$, we denote $n_j = p^{\sum_{i=1}^{j}m_i\alpha_i}$, $r_j = \sum_{i=1}^{j} m_i$, and that $n=n_s$ and $r = r_s$. Additionally, we write $n_0 = 1$ and $r_0 = 0$ for consistency.
\end{notation}

We derive the independence number of the power graph of a finite abelian $p$-group in the next theorem. To state as well as to prove it, we follow Notation \ref{not1} throughout.

 As every finite abelian group is a direct product of cyclic groups of prime power order, assume in the following that $G \cong \mathbb{Z}^{m_1}_{p^{\alpha_1}} \times \cdots \times \mathbb{Z}^{m_s}_{p^{\alpha_s}}$ for any finite abelian $p$-group $G$.

\begin{theorem}\label{indep.number}
	For any finite abelian $p$-group $G$, the independence number
	\begin{align*}
	\alpha(\mathcal{P}(G)) &  = \sum_{t=1}^{r} \dfrac{n}{p^{t-1}}  \left\{ \frac{p^{(r_k-1)\alpha_k}}{n_k} + \mu_k \right\},
	\end{align*}
where
\begin{flushright}
$
\mu_k =
     \begin{cases}
       \sum_{j=1}^{k-1}  \dfrac{ p^{r_j} -1 }{n_j} \sum_{\beta = \alpha_{j+1}}^{\alpha_j-1} p^{(r_j-1)\beta} \text{ if }k>1,\\
       0 \text{ if } k=1,
     \end{cases}
$
\end{flushright}
and $k$ is such that $r_{k-1} < t \leq r_{k}$.
\end{theorem}

\begin{proof}
	It enough to prove for $G = \mathbb{Z}^{m_1}_{p^{\alpha_1}} \times \cdots \times \mathbb{Z}^{m_s}_{p^{\alpha_s}}$.
	Notice that $|G| = n$ and $G$ is direct a product of $r$ cyclic subgroups.
	If $G$ is cyclic, then $\alpha(\mathcal{P}(G)) = 1$ since $\mathcal{P}(G)$ is complete. Now let $G$ be noncyclic, that is, $r \geq 2$.
	
We denote the set of all maximal cyclic subgroups of $G$ by $\mathcal{M} = \mathcal{M}(G)$. Then by Theorem \ref{indep}, $\alpha(\mathcal{P}(G)) = |\mathcal{M}|$. Now, to compute $|\mathcal{M}|$, we shall partition $\mathcal{M}$ as follows.
For $x \in G$, we denote the $i$th component of $x$ by $x_i$. For any maximal cyclic subgroup $\langle x \rangle$ of $G$, observe that $(x_i,p)=1$ for at least one $1 \leq i \leq r$. For any $1 \leq t \leq r$, we define $\mathcal{M}_t = \{ \langle x \rangle \in \mathcal{M} : (x_t,p)=1 \text{ and } (x_1,p) \neq 1, \dots, (x_{t-1},p) \neq 1 \text{ if } t>1 \}$.  Thus $\alpha(\mathcal{P}(G)) = \sum_{t=1}^{r}  |\mathcal{M}_t|$.


Now we fix $1 \leq t \leq r$ and compute $|\mathcal{M}_t|$. We have $t = r_{k-1} + m$ for some $1 \leq k \leq s$ and  $1 \leq m \leq m_{k}$. So that if $\langle x \rangle \in \mathcal{M}_t$, then $o(x) \geq p^{\alpha_k}$. Next for any $\beta \geq \alpha_k$, we define $\mathcal{M}_{t,p^\beta} = \{ \langle x \rangle \in \mathcal{M}_t : o(x) = p^\beta \}$. Then we have,
		\begin{align*}
		&|\{x : \langle x \rangle \in \mathcal{M}_{t,p^{\alpha_k}} \}|\\
		& =  p^{{r_{k-1}}\alpha_k}  \cdot p^{{(m-1)}(\alpha_{k}-1)} \cdot \phi(p^{\alpha_k}) \cdot \dfrac{n}{n_{k-1} \cdot p^{m\alpha_{k}}}\\
		& = \dfrac{n\phi(p^{\alpha_k})}{p^{\{\sum_{i=1}^{k}{m_i}(\alpha_i-\alpha_{k})\}+\alpha_k+m-1}}.
		\end{align*}
		
		Hence,
\begin{align}\label{indepeq3}
|\mathcal{M}_{t,p^{\alpha_k}}| = \dfrac{n}{p^{\{\sum_{i=1}^{k}{m_i}(\alpha_i-\alpha_{k})\}+\alpha_k+m-1}}.
\end{align}		

Suppose that $t \leq m_1$. Then $k=1$ and that $t = m$. Moreover, $o(x) = p^{\alpha_{1}}$ for any $\langle x \rangle \in \mathcal{M}_{t}$, so that $\mathcal{M}_{t} = \mathcal{M}_{t,p^{\alpha_k}}$. Thus we have
		$$|\mathcal{M}_{t}| = \dfrac{n}{p^{\alpha_1+t-1}} = \dfrac{n}{p^{t-1}}  \left\{ \frac{p^{(r_1-1)\alpha_1}}{n_1} \right\}.$$
		Consequently, the proof follows for $t \leq m_1$.
		
For remaining of the proof, we take $t > m_1$, that is, $k > 1$.
We observe that $p^{\alpha_k} \leq o(x) \leq p^{\alpha_1-1}$ for any $\langle x \rangle \in \mathcal{M}_t$.
		 Let $\beta > \alpha_k$ be such that $\alpha_{l+1} \leq \beta \leq  \alpha_l-1$ for some $1 \leq l \leq k-1$. Then
		
		\begin{align*}
		& |\{x : \langle x \rangle \in \mathcal{M}_{t,p^{\beta}} \}|\\
		& = \{ \phi(p^\beta) (p^\beta)^{m_1+\cdots+m_{l}-1} + p^{\beta-1} \phi(p^\beta) (p^\beta)^{m_1+\cdots+m_{l}-2} + \cdots + (p^{\beta-1})^{m_1+\cdots+m_{l}-1} \phi(p^\beta)  \}\\
		&\qquad \qquad \qquad \qquad \qquad \cdot p^{{m_{l+1}}(\alpha_{l+1}-1)} \cdots p^{{m_{k-1}}(\alpha_{k-1}-1)} \cdot \dfrac{n  \cdot p^{{(m-1)}(\alpha_{k}-1)} \cdot \phi(p^{\alpha_k})}{p^{{m_1}\alpha_1} \cdots p^{{m_{k-1}}\alpha_{k-1}} \cdot p^{m\alpha_{k}}}\\		
		& = p^{{(m_1+\cdots+m_{l}-1)({\beta-1})}} \left ( \dfrac{p^{m_1+\cdots+m_{l}}-1}{p-1} \right )\\
		&\qquad \qquad \qquad \qquad \qquad \cdot p^{{m_{l+1}}(\alpha_{l+1}-1)} \cdots p^{{m_{k-1}}(\alpha_{k-1}-1)} \cdot \dfrac{n\cdot p^{{(m-1)}(\alpha_{k}-1)} \cdot \phi(p^{\alpha_k}) \cdot  \phi(p^\beta)}{p^{{m_1}\alpha_1} \cdots p^{{m_{k-1}}\alpha_{k-1}} \cdot p^{m\alpha_{k}}}\\
		& = p^{{(m_1+\cdots+m_{l}-1)({\beta-1})}} \left (p^{m_1+\cdots+m_{l}}-1\right )\\
		&\qquad \qquad \qquad \qquad \qquad \cdot p^{{m_{l+1}}(\alpha_{l+1}-1)} \cdots p^{{m_{k-1}}(\alpha_{k-1}-1)} \cdot \dfrac{n\cdot p^{{m}(\alpha_{k}-1)} \cdot \phi(p^\beta)}{p^{{m_1}\alpha_1} \cdots p^{{m_{k-1}}\alpha_{k-1}} \cdot p^{m\alpha_{k}}}\\
		& = \dfrac{n\phi(p^\beta) \left( p^{\sum_{i=1}^{{l}}{m_i}} -1 \right)}{p^{\{\sum_{i=1}^{l}{m_i}(\alpha_i-\beta)\}+\sum_{i=1}^{k}{m_i}-m_k+\beta+m-1}}.
		\end{align*}
		
		Hence,
		\begin{align*}
		|\mathcal{M}_{t,p^\beta}| = \dfrac{n \left( p^{\sum_{i=1}^{{l}}{m_i}} -1 \right)}{p^{\{\sum_{i=1}^{l}{m_i}(\alpha_i-\beta)\}+\sum_{i=1}^{k}{m_i}-m_k+\beta+m-1}}.
		\end{align*}
		
 From \eqref{indepeq3}, we get
		
		$$|\mathcal{M}_{t,p^{\alpha_k}}| = \dfrac{n}{p^{\sum_{i=1}^{k}{m_i}(\alpha_i-\alpha_{k})+\sum_{i=1}^{k-1}{m_i}+\alpha_k+m-1}}  + \dfrac{n\left( p^{\sum_{i=1}^{{k-1}}{m_i}} -1 \right)}{p^{\sum_{i=1}^{k}{m_i}(\alpha_i-\alpha_{k})+\sum_{i=1}^{k-1}{m_i}+\alpha_k+m-1}} .$$
	
		Additionally,	
\begin{align*}
& \sum_{j=1}^{k-1} \sum_{\beta = \alpha_{j+1}}^{\alpha_j-1} \dfrac{n \left( p^{\sum_{i=1}^{j}{m_i}} -1 \right)}{p^{\sum_{i=1}^j{m_i}(\alpha_i-\beta)+\sum_{i=1}^{k-1}{m_i}+\beta+m-1}}\\
& =  \sum_{j=1}^{k-1}  \dfrac{n \left( p^{\sum_{i=1}^{j}{m_i}} -1 \right)}{p^{\sum_{i=1}^j{m_i}\alpha_i +\sum_{i=1}^{k-1}{m_i}+m-1}} \sum_{\beta = \alpha_{j+1}}^{\alpha_j-1} p^{(\sum_{i=1}^j{m_i}-1)\beta}\\
& = \dfrac{n}{p^{\sum_{i=1}^{k-1}{m_i}+m-1}} \sum_{j=1}^{k-1}  \dfrac{ p^{\sum_{i=1}^{j}{m_i}} -1 }{p^{\sum_{i=1}^j{m_i}\alpha_i}} \sum_{\beta = \alpha_{j+1}}^{\alpha_j-1} p^{(\sum_{i=1}^j{m_i}-1)\beta}.
\end{align*}

Therefore, we have
\begin{align*}
|\mathcal{M}_{t}|
 & = \sum_{j=1}^{k-1} \sum_{\beta = \alpha_{j+1}}^{\alpha_j-1} |\mathcal{M}_{t,p^{\beta}}| \\
  & = \dfrac{n}{p^{\sum_{i=1}^{k}{m_i}(\alpha_i-\alpha_{k})+\sum_{i=1}^{k-1}{m_i}+\alpha_k+m-1}} + \sum_{j=1}^{k-1} \sum_{\beta = \alpha_{j+1}}^{\alpha_j-1} \dfrac{n \left( p^{\sum_{i=1}^{j}{m_i}} -1 \right)}{p^{\sum_{i=1}^j{m_i}(\alpha_i-\beta)+\sum_{i=1}^{k-1}{m_i}+\beta+m-1}}\\
& = \dfrac{n}{p^{\sum_{i=1}^{k-1}{m_i}+m-1}}  \left\{ \dfrac{1}{p^{\sum_{i=1}^{k}{m_i}(\alpha_i-\alpha_{k})+\alpha_k}}
 + \sum_{j=1}^{k-1}  \dfrac{ p^{\sum_{i=1}^{j}{m_i}} -1 }{p^{\sum_{i=1}^j{m_i}\alpha_i}} \sum_{\beta = \alpha_{j+1}}^{\alpha_j-1} p^{(\sum_{i=1}^j{m_i}-1)\beta} \right\}\\
& = \sum_{t=1}^{r} \dfrac{n}{p^{t-1}}  \left\{ \frac{p^{(r_k-1)\alpha_k}}{n_k} + \sum_{j=1}^{k-1}  \dfrac{ p^{r_j} -1 }{n_j} \sum_{\beta = \alpha_{j+1}}^{\alpha_j-1} p^{(r_j-1)\beta} \right\}.
\end{align*}

This concludes the proof of the theorem.	
\end{proof}

Since every abelian group is nilpotent, from Lemma \ref{nilpotent} and Theorem \ref{indep.number}, we have the following corollary.

\begin{corollary}\label{indep.abelian}
	For any finite abelian group $G$ with the Sylow subgroups $P_1,P_2, \dots, P_r$  of $G$, the independence number  $\alpha(\mathcal{P}_e(G)) = \prod_{i=1}^{r} \mu(P_i)$, where $\mu(P_i)$ can be computed using Theorem \ref{indep.number} for $1 \leq i \leq r$.
\end{corollary}

The following theorem computes the the matching number of enhanced power graphs of finite groups.

\begin{theorem}\label{matching.number}
	For any finite $p$-group $G$, the matching number
\begin{equation*}
\alpha'(\mathcal{P}(G)) =
\begin{cases}
\dfrac{|G|-1}{2} & p > 2,\\
\dfrac{|G|-(t-1)}{2}  & p = 2,
\end{cases}
\end{equation*}
where $t$ is the number of involutions in $G$.
\end{theorem}

\begin{proof}
	For $p > 2$, the result follows from Theorem \ref{matching}. So for the rest of the proof, we take $p=2$.
	
	Observe that the only common vertex between any two distinct blocks in $\mathcal{P}(G)$ is $e$. Thus the number of blocks in $\mathcal{P}(G)$ coincides with the number of components of $\mathcal{P}(G {\setminus} \{e\})$.
	Following Lemma \ref{Componentp}, every block in $\mathcal{P}(G)$ has exactly one vertex of order two.
	So that the number of blocks in $\mathcal{P}(G)$ coincides with the number of involutions, say  $t$, in $G$.
	Moreover, in light of Lemma \ref{phieven}, the number of vertices in any block in $\mathcal{P}(G)$ is even.
	
	Consider a matching $M$ in $\mathcal{P}(G)$. Let $M'$ be the subset of $M$ containing all elements of $M$ whose endpoints are nonidentity elements of $G$.  Let $B$ be any block in $\mathcal{P}(G)$. Since $|V(B)|$ is even, $E(B)$ contains atmost $\dfrac{V(B)-2}{2}$ elements of $M'$. From this and the fact that $e$ is a vertex in every block in $\mathcal{P}(G)$, we have  $|M'| \leq \dfrac{|G|-(t+1)}{2}$. Additionally, $M$ contains at most one edge with $e$ as an endpoint. Thus we have $|M| \leq |M'|+1 \leq \dfrac{|G|-(t-1)}{2}$.
	
	Therefore, to prove the theorem, it is enough to produce a matching of cardinality $\dfrac{|G|-(t-1)}{2}$.
	Let $\epsilon$ be an edge with $e$ as one end and the other an involution.
	Recall the definition of $M_x$ for any $x \in G$.
	We consider $M_G:= \{\epsilon\} \bigcup \displaystyle\cup_{x \in S, \; o(x)>2}M_x$,
	where $S$ is a subset of $G$ containing exactly one element from each $\approx$-class. Then all elements of $M_G$, except $\epsilon$, have both ends in same $\approx$-class, and $\epsilon$ does not have common endpoints with any other element $M_G$. Hence $M_G$ is a matching in $\mathcal{P}(G)$.  Finally, as $|M_G| = \dfrac{|G|-(t+1)}{2}+1 = \dfrac{|G|-(t-1)}{2}$, the proof follows.
\end{proof}

\begin{theorem}\label{sdm1}
	Let $G$ be any finite abelian $p$-group with exponent $p^\alpha$. Then the strong metric dimension of $\mathcal{P}_e(G)$ is
	\begin{enumerate}[\rm(i)]
		\item $|G|-(\alpha+1)$ if $G$ is non-cyclic,
		\item $|G|-1$ if $G$ is cyclic.
	\end{enumerate}

\end{theorem}

\begin{proof}
	The proof is straightforward when $G$ is cyclic. Now let $G$ be noncyclic. Then $G \cong C_{1} \times \cdots \times C_{r}$ for some cyclic $p$-groups $C_{1}, \dots, C_{r}$, $r \geq 2$. Consider two distinct elements $x = (x_1, x_2, \dots, x_r)$ and $y = (y_1, y_2, \dots, y_r)$ in $G$ with $o(x) \geq o(y)$. Suppose that $N[x] = N[y]$. Clearly, $x \sim y$.
	
	 If $y=e$, we have $N[x] = N[y] = G$. However, this is not possible in view of Lemma \ref{cameron}. Analogous situation occurs when $x = e$. Hence $x$ and $y$ are nonidentity elements.
	
	If possible, let $o(x)=2$. Then $o(y)=2$, since $o(x) \geq o(y)$. As $x \sim y$, we thus have $\langle x \rangle = \langle y \rangle$. That is, $x = y$, which is a contradiction.
	
	So that $o(x) \geq 3$.
	Now suppose $o(x) > o(y)$. Then as $x \sim y$, there exists an integer $t$ with $p \, | \, t$ such that $y=x^t$. We have the following cases.
	
	\smallskip
	\noindent
	{\bf Case 1.} $x_i = e$ for some fixed $1 \leq i \leq r$. We define an element ${z} = ({z_1}, {z_2}, \dots, {z_r})$ such that ${z_i}$ is an element of order $p$ in $C_i$ and ${z_j} = x_j$ for $1 \leq j \leq r$, $j \neq i$. Then notice that $o({z}) = o(x)$ and $\langle z \rangle \neq \langle x \rangle$. As a result, ${z} \nsim x$. However, as  $y={z}^t$, we have  ${z} \sim y$. This contradicts our assumption that $N[x] = N[y]$.
	
	\smallskip
	\noindent	
	{\bf Case 2.} $x_j \neq e$ for all $1 \leq j \leq r$. Let $1 \leq k \leq r$ be such that $o(x_k) \leq o(x_j)$ for all $1 \leq j \leq r$. We define an element ${w} = ({w_1}, {w_2}, \dots, {w_r})$ such that ${w_k} = {x_k}^{p^{\beta-1}+1}$, where $o(x_k) = p^{\beta}$, and ${w_j} = x_j$ for $1 \leq j \leq r$, $j \neq k$. Then $y={w}^t$, so that ${w} \sim y$. Moreover, as $o({w}) = o(x)$, we have $w = x^s$ for some $(s,p)=1$. Accordingly, $o(x_i) \, | \, (s-1)$ for all $1 \leq j \leq r$, $j \neq k$. Since $o(x_k) \leq o(x_j)$ for all $1 \leq j \leq r$, we thus get $p^{\beta} {\, | \,} (s-1)$. This implies ${x_k}^{p^{\beta-1}+1} = {x_k}$, which is not possible. As a result, ${w} \nsim x$. This again results in a contradiction.
	
	Consequently, $o(x) = o(y)$. Hence as $x \sim y$, we have $\langle x \rangle = \langle y \rangle$.
	
	Since converse is trivial, we therefore conclude that $N[x] = N[y]$ if and only if $o(x) \geq 3$ and $\langle x \rangle = \langle y \rangle$.
	From Lemma \ref{cameron}, we get $\widehat{e} = \{e\}$. Moreover, if $p=2$, then $\widehat{x} = \{x\}$ for every element $x$ of order $2$ in $G$. Hence the $\equiv$-classes and $\approx$-classes coincide for every $x \in G$.
	
	 Now consider a clique $C$ in $ \widehat{\mathcal{P}}(G) $ with at least two vertices. Then for any pair of distinct vertices $\widehat{x}$, $\widehat{y}$ in $C$, we have $o(x) \neq o(y)$. Additionally, for any $x \in G$, we have $o(x)  = p^i$ for some $0 \leq i \leq \alpha$. Thus $\omega(\widehat{\mathcal{P}}(G)) \leq \alpha+1$. Since $G$ is $p$-group of exponent $p^\alpha$, there exists $z \in G$ of order $p^\alpha$. We observe that $\{\widehat{e}\} \cup \{\widehat{z^{p^{i}}} : 0 \leq i \leq \alpha-1\}$ is a clique in $\widehat{\mathcal{P}}(G)$. Therefore, we get $\omega(\widehat{\mathcal{P}}(G)) = \alpha+1$ and subsequently,   $\operatorname{sdim}({\mathcal{P}}(G)) = |G|-(\alpha+1)$, by Theorem \ref{strong-metric-dim}.
	
\end{proof}

\bigskip
\subsection{The group $U_{6n}$.}\hfill
\smallskip

For $n \geq 1$, the group $U_{6n}$ of order $6n$ is given by the presentation
$$U_{6n} = \langle a, b : a^{2n} = b^3 = e, ba =ab^{-1} \rangle.$$

We first study the structure of $U_{6n}$ and then investigate properties of $\mathcal{P}_e(U_{6n})$.

\begin{remark}\label{order-U_6n}
	The  group $U_{6n}$ is  of order $6n$ if and only if $b \notin \langle a \rangle$.
\end{remark}
Since $ba =ab^{-1}$, for any $0 \leq i \leq  2n-1$, we have
$$ba^i = \left\{ \begin{array}{ll}
a^ib & \mbox{if $i$ is even,}\\
a^ib^2& \mbox{if $i$ is odd,}\end{array} \right. \;
\text{and} \; \;  \; \; \; \; \;
b^2a^i = \left\{ \begin{array}{ll}
a^ib^2 & \mbox{if $i$ is even,}\\
a^ib & \mbox{if $i$ is odd.}\end{array} \right.$$
\\
Thus every element of $U_{6n}{\setminus}\langle a \rangle$ is of the form $a^ib^j$ for some $0 \leq i \leq  2n-1$ and $1 \leq j \leq 2$. \\
Moreover, for any $0 \leq i \leq  2n-1$,
\begin{align}\label{u6neq}
	(ab)^i  =  \left\{ \begin{array}{ll}
	a^i & \mbox{if $i$ is even,}\\
	a^{i}b & \mbox{if $i$ is odd,}\end{array} \right. \; \; \text{and} \; \;  \; \; (ab^2)^i  =  \left\{ \begin{array}{ll}
	a^i & \mbox{if $i$ is even,}\\
	a^{i}b^2 & \mbox{if $i$ is odd.}\end{array} \right.
\end{align}

Consequently, we have the following remarks.

\begin{remark}\label{odd-power-ab}
	For  $x \in U_{6n}$, we have $x = a^{2s + 1}b$ for some $0 \leq s \leq n-1$ if and only if $x \in \langle ab \rangle {\setminus} \langle a \rangle$.	
\end{remark}

\begin{remark}\label{odd-power-ab^2}
	For  $x \in U_{6n}$, we have $x = a^{2s + 1}b^2$ for some $0 \leq s \leq n-1$ if and only if $x \in \langle ab^2 \rangle {\setminus} \langle a \rangle$.	
\end{remark}

\begin{remark}\label{U_6n-minus-ab-ab^2}
	Every element of $U_{6n} {\setminus} \left(\langle a \rangle \cup \langle ab \rangle \cup \langle ab^2 \rangle\right)$ is   of the form $a^{2i}b$ and $a^{2j}b^2$ for some  $i, j$.
\end{remark}

From the presentation of $U_{6n}$ and by  mathematical induction, we have

\begin{align}\label{u6neq1}
	(a^{2\cdot 3^{i}}b)^j  =  \left\{ \begin{array}{ll}
	a^{2\cdot 3^{i}\cdot j} & \mbox{if $j \equiv  0$ (mod 3),}\\
	a^{2\cdot 3^{i}\cdot j}b & \mbox{if $j \equiv  1$ (mod 3),}\\
	a^{2\cdot 3^{i}\cdot j}b^2 & \mbox{if $j \equiv  2$ (mod 3),}\end{array} \right.   \; \; \text{and} \; \; \; \; \; \; \;
	(a^{2\cdot 3^{i}}b^2)^j  =  \left\{ \begin{array}{ll}
	a^{2\cdot 3^{i}\cdot j} & \mbox{if $j \equiv  0$ (mod 3),}\\
	a^{2\cdot 3^{i}\cdot j}b^2 & \mbox{if $j \equiv  1$ (mod 3)},\\
	a^{2\cdot 3^{i}\cdot j}b & \mbox{if $j \equiv  2$ (mod 3)}.\end{array} \right.
\end{align}

In the rest of this section, we shall write  $n = 3^kt$  for integers $k \geq 0$ and $t>0$ such that $3 \nmid t$.

By Remark \ref{U_6n-minus-ab-ab^2} and \eqref{u6neq1}, we have the following lemma.

\begin{lemma}\label{even-power-a-b}
	For $n \geq 1$, we have  $U_{6n} = \langle a \rangle \cup \langle ab \rangle \cup \langle ab^2 \rangle \cup  \bigcup\limits_{i = 0}^{k} \left(\langle a^{2\cdot3^{i}}b\rangle \cup \langle a^{2\cdot3^{i}}b^2\rangle\right).$
\end{lemma}

For the remainder of this subsection, we shall denote
 $P_i = \langle a^{2\cdot 3^{i}}b \rangle$, $Q_i = \langle a^{2\cdot 3^{i}}b^2 \rangle$ for $0 \leq i \leq k$, and $P_{k +1 } = \langle ab \rangle$, $Q_{k+1} = \langle ab^2 \rangle$.

 Thus we have
 \begin{align}\label{u6neq2}
 	U_{6n} =  \langle a \rangle \cup \bigcup_{i=0}^{ k}\left( P_i \cup Q_i\right) \cup \left(P_{k+1} \cup Q_{k+1}\right).
 \end{align}

We observe that that $\{ b, a^{2\cdot 3^{k}} \} \subset \langle a^{2\cdot 3^{k}}b \rangle \cap \langle a^{2\cdot 3^{k}}b^2 \rangle$. As consequences, we have the following remarks.
\begin{remark}
	For $i = k$, we have $P_i = Q_i$.
\end{remark}

\begin{remark}\label{a^{23^i+1}-a^{23^i}}
	For $0 \leq i \leq k-1$, we have $\langle a^{2\cdot 3^{{i + 1}}} \rangle \subset \langle a^{2\cdot 3^{{i}}} \rangle. $
\end{remark}

\begin{remark}
	For $i = k$, we have $\langle a^{2\cdot 3^{{i  + 1}}} \rangle =  \langle a^{2\cdot 3^{{i}}} \rangle. $
\end{remark}

Since $o(a) = 2n$, we have the following remark.
\begin{remark}\label{order-a^{23^k}}
	In $U_{6n}$, we have $o(a^{2\cdot3^{k}}) = t$.
\end{remark}

In view of \eqref{u6neq} and \eqref{u6neq1}, we have the following lemma.

\begin{lemma}\label{P_i-intersection-generated-a}
	For $i \leq k$, we have $P_i \cap \langle a \rangle =  Q_i \cap \langle a \rangle = \langle a^{2\cdot 3^{i +1}} \rangle$. Moreover,  $P_k \cap \langle a \rangle = \langle a^{2\cdot 3^{k}} \rangle$.
\end{lemma}
\begin{proof}
	By \eqref{u6neq1}, we have $(a^{2\cdot 3^{i}}b)^j \in P_i \cap \langle a \rangle$ if and only if $j = 3l$ for some $l$. Then  $ P_i \cap \langle a \rangle = \langle (a^{2\cdot 3^{i}}b)^3 \rangle = \langle a^{2\cdot 3^{{i +1}}} \rangle$. Similarly, we have $Q_i \cap \langle a \rangle = \langle a^{2\cdot 3^{i +1}} \rangle$. For $i =k$, we have $P_k \cap \langle a \rangle = \langle a^{2\cdot 3^{{k +1}}} \rangle$. Since $a^{2\cdot 3^{{k+1}}} \in \langle a^{2\cdot 3^{{k}}} \rangle$, and
	$$a^{2\cdot 3^{k}}  =  \left\{ \begin{array}{ll}
	\left( a^{2\cdot 3^{{k +1}}} \right)^{\frac{2t +1}{3}} & \mbox{if $t \equiv  1$ mod 3},\\
	\left( a^{2\cdot 3^{{k +1}}}\right)^{\frac{t +1}{3}}& \mbox{if $t \equiv  2$ mod 3}
	,\end{array}\right.$$
	we get $P_k \cap \langle a \rangle = \langle a^{2\cdot 3^{{k}}} \rangle$.
\end{proof}

\begin{lemma}\label{order-element-U_6n}
	For the group  $U_{6n}$, we have
	\begin{enumerate}[\rm(i)]
		\item $|P_i| = |Q_i| = 3^{k - i}t$, where  $0 \leq i \leq k - 1$. Moreover, $|P_i {\setminus} \langle a \rangle| = |Q_i {\setminus} \langle a \rangle| = 2 \cdot 3^{k - i - 1}t$
		\item $|P_k| = |Q_k| = 3t$. Moreover, $|P_k {\setminus} \langle a \rangle| = |Q_k {\setminus} \langle a \rangle| = 2t$.
		\item $|P_{k + 1}| = |Q_{k +1}| = 2n$. Moreover, $|P_{k +1} {\setminus} \langle a \rangle| = |Q_{k + 1} {\setminus} \langle a \rangle| =n$.
	\end{enumerate}
\end{lemma}

\begin{proof}
	(i) Since $P_i =  \langle a^{2\cdot 3^{{i}}}b \rangle$, we have $ (a^{2\cdot 3^{{i}}}b)^{3^{{k-i}} t} = e$. If $l <   3^{{k-i}} t$, then by \eqref{u6neq1} and Remark \ref{order-U_6n},  $(a^{2\cdot 3^{i}}b)^l \neq e$. So that $|P_i| = 3^{k - i}t$. Now we observe that $o(a^{2\cdot 3^{{i + 1}}}) = 3^{{k -i -1 }} t$. Thus we get $|P_i {\setminus} \langle a \rangle| = 2\cdot 3^{{k - i - 1}} t$, by Lemma \ref{P_i-intersection-generated-a}. Similarly, $|Q_i {\setminus} \langle a \rangle| = 2\cdot 3^{{k - i - 1}} t$.
	
	\smallskip
	\noindent	
	(ii) Since $P_k =  \langle a^{2\cdot 3^{k}}b \rangle$, by  \eqref{u6neq1}, we have $(a^{2\cdot3^{k}}b)^t \in \{ b,b^2 \}$ and  $ a^{2\cdot3^{k}}\in \{(a^{2\cdot3^{k}}b)^{2t + 1}, (a^{2\cdot3^{k}}b)^{t+1} \}$. Consequently, $a^{2\cdot3^{k}}, b \in \langle a^{2\cdot3^{k}}b \rangle$. By Remark \ref{order-a^{23^k}} and as $o(a^{2\cdot 3^{k}}) = t$, it follows that  $(a^{2\cdot 3^{k}})^{i}b, (a^{2\cdot 3^{k}})^{j}b^2$ and $(a^{2\cdot 3^{k}})^l$ are all distinct elements in   $\langle a^{2\cdot3^{k}}b \rangle$, where $1 \leq i,j, l \leq t$. As a result, $ o(a^{2\cdot3^{k}}b) \geq 3t$.  Since $(a^{2\cdot3^{k}}b)^{3t} = e$, we obtain  $o(a^{2\cdot3^{k}}b) = 3t = |P_k|$. Accordingly, by Lemma \ref{P_i-intersection-generated-a} and Remark \ref{order-a^{23^k}},  $|P_k {\setminus} \langle a \rangle| = 2t$. The proof for $Q_k$ is similar.
	
	\smallskip
	\noindent		
	(iii) The proof is straightforward following \eqref{u6neq}.
\end{proof}

The next result describes structure of $U_{6n}$ further.

\begin{proposition}\label{Common-elementP_i-Q_i}
	For $0 \leq i \leq k+1$, the following hold:
	\begin{enumerate}[\rm(i)]
		\item If $x \in P_i {\setminus} \langle a \rangle$,  then $x \notin \left( \displaystyle \bigcup_{j=0, j\neq i}^{k+1} P_j \right) \cup \left(\displaystyle  \bigcup_{j=0}^{k+1} Q_j \right)$.
		\item If $x \in Q_i {\setminus} \langle a \rangle$,  then $x \notin \left(\displaystyle  \bigcup_{j=0}^{k+1} P_j \right) \cup \left(\displaystyle  \bigcup_{j=0, j\neq i}^{k+1} Q_j \right)$.
	\end{enumerate}
\end{proposition}

\begin{proof}
	\noindent
	(i) If possible, let $x \in U_{6n} {\setminus} P_i$. Then in view of \eqref{u6neq2}, we have the following cases.
	
	\smallskip
	\noindent
	{\bf Case 1.} $ x \in Q_i$.
	For $i = k+1$, we have $x \in \langle ab \rangle \cap \langle ab^2 \rangle$. By Remark \ref{odd-power-ab} and Remark \ref{odd-power-ab^2}, we have $x = a^{2r +1}b$ and $x = a^{2s +1}b^2$ for some $r,s$. Consequently, $b \in \langle a \rangle$, a contradiction of Remark \ref{order-U_6n}. Suppose $0 \leq i \leq k$. Then by \eqref{u6neq2},  $x = (a^{2\cdot 3^{i}} b)^p$ and  $x = (a^{2\cdot 3^{i}} b^2)^q$, where $p, q \leq 3^{k -i }t$ (see Lemma \ref{order-element-U_6n}(i)). Clearly, $3 \nmid p$ and $3 \nmid q$. Otherwise, $x \in \langle a \rangle$, a contradiction. If   $p \equiv1($mod $ 3)$ and $q\equiv 2($mod $3)$, then by \eqref{u6neq1},  $x = a^{{2\cdot 3^{i}\cdot p}}b$ and $x = a^{{2\cdot 3^{i}\cdot q}}b$. Since  $p, q \leq 3^{k -i }t$, we have  $2\cdot 3^i\cdot p \leq 2n$ and   $2\cdot 3^i\cdot q \leq 2n$. Consequently,   $a^{{2\cdot 3^i\cdot p}}  b =  a^{{2\cdot 3^i\cdot q}} b$ gives $p = q$, which is not  possible. Similarly, we get $p = q$ for the case   $p \equiv 2($mod $ 3)$ and  $q \equiv 1($mod $3)$, again a contradiction. For the case  $p \equiv1($mod $ 3)$ and $ q \equiv 1($mod $3)$, we have $x = a^{{2\cdot 3^{i}\cdot p}}b$ and $x = a^{{2\cdot 3^{i}\cdot q}}b^2$. Consequently, $b \in \langle a \rangle$, a contradiction. We get a similar contradiction for the case  $p \equiv2($mod $ 3)$ and $ q \equiv 2($mod $3)$.
	
	\smallskip
	\noindent		
	{\bf Case 2.}  $x \in P_j$ {\rm with} $j \neq i$.
	If  $j = k+1$, then by Remark \ref{odd-power-ab}, we have  $x = a^lb$ for some odd $l$. Since $j \neq i$, we get either  $x= a^{l'}b$ or $x= a^{l'}b^2$ for some even $l'$, a contradiction. Similarly, we get a contradiction when $i = k + 1$. Let $0 \leq i, j \leq k$. Then  $x = (a^{2\cdot 3^{i}} b)^u $ and $x = (a^{2\cdot 3^{j}} b)^v$ for some $u \leq 3^{k -i }t$ and $v\leq 3^{k -j }t$ (see Lemma \ref{order-element-U_6n} (i)). Clearly, $u, v$ are not divisible by $3$. Otherwise, $x \in \langle a \rangle$, a contradiction. If $u \equiv1($mod $ 3)$ and $v \equiv 2($mod $3)$, then by \eqref{u6neq1}, we get $x = a^{{2\cdot 3^{i}\cdot u}} b $ and $x = a^{{2\cdot 3^{j}\cdot v}} b^2$. Consequently, $b \in  \langle a \rangle$, a contradiction of Lemma \ref{order-U_6n}. Similarly, we get a contradiction if $u \equiv 2 ($mod $3)$ and  $ v \equiv 1($mod $3)$. For the case $u \equiv 1 ($mod $3)$ and $v \equiv 1 ($mod $3)$, we have  $x = a^{{2\cdot 3^{i}\cdot u}}b$ and $x = a^{{2\cdot 3^{j}\cdot v}}b$. Since $u \leq 3^{k -i }t$ and $v\leq 3^{k -j }t$, we have $2\cdot 3^i\cdot  u \leq 2n$ and $2\cdot 3^j \cdot v \leq 2n$. Consequently,   $a^{{2\cdot 3^i\cdot u}}  b =  a^{{2\cdot 3^j\cdot v}} b$ gives $2\cdot 3^i \cdot u = 2\cdot3^j\cdot v$. Without loss of generality, we assume that $i < j$. Now, we get $u  = 3^{j -i}v$ implies $3 \, | \, u$, a contradiction. Similarly, we  arrive at a contradiction if $u\equiv2$(mod $ 3)$ and $v\equiv2$(mod $ 3)$.
	
	\smallskip
	\noindent
	{\bf Case 3.}  $ x \in Q_j$ {\rm with} $j \neq i$.
	If $j = k + 1$, then by Remark \ref{odd-power-ab^2}, we get $x = a^{m}b^2 $ for some odd $m$. Since  $i \neq j$, we have either $x = a^{m'}b$ or $x = a^{m'}b^2$  for some even $m'$ which is not possible. Similarly, we have a contradiction if $i = k + 1$. So, we assume that $0\leq i, j \leq k$. Then $x = (a^{2\cdot 3^{i}} b)^{r'} $ and $x =  (a^{2\cdot 3^{j}} b^2)^{s'}$ for some $r' \leq 3^{k -i }t$ and $s'\leq 3^{k -j }t$ (see Lemma \ref{order-element-U_6n} (i)). Clearly, $r', s'$ are not divisible by $3$.  Otherwise, $x \in \langle a \rangle$, a contradiction. For the case   $r' \equiv 1($mod $3)$ and $s' \equiv 2($mod $3)$, we have $x = a^{2 \cdot 3^i\cdot r'}b$ and $x =a^{2 \cdot 3^j\cdot s'}b$. Since $r' \leq 3^{k -i }t$ and $s'\leq 3^{k -j }t$, we have $2\cdot 3^i\cdot  r' \leq 2n$ and $2\cdot 3^j \cdot s' \leq 2n$. Consequently   $a^{{2\cdot 3^i\cdot r'}}  b =  a^{{2\cdot 3^j\cdot s'}} b$ gives $2\cdot 3^i \cdot r' = 2\cdot3^j\cdot s'$. Without loss of generality, we assume that $i < j$. Now, we get $r'  = 3^{j -i}s'$ implies $3 \, | \, r'$, a contradiction. Similarly, we get $3 \, | \, r'$ if $r' \equiv2$(mod $ 3)$ and $s' \equiv1$(mod $ 3)$.  For   $r' \equiv 1 ($mod $3)$ and  $s' \equiv 1($mod $3)$, we have $x = a^{2\cdot 3^i \cdot r'}b $ and $x =  a^{2\cdot 3^j \cdot s'}b^2$. Consequently,  $b^2 \in \langle a \rangle$, again a contradiction. We get a similar contradiction when $r' \equiv 2 ($mod $3)$ and $s' \equiv 2($mod $3)$.
	
	\smallskip
	\noindent
	(ii) The proof  is similar to that of (i).
\end{proof}

As a consequence of Proposition \ref{Common-elementP_i-Q_i}, we have the following lemma.

\begin{lemma}\label{degree-element-P_i}
	Let $x \in P_i {\setminus} \langle a \rangle$, where $0 \leq i \leq k +1$. Then $x \sim y$ for any $y \in U_{6n}$ if and only if $y \in P_i$.
\end{lemma}

\begin{proof}
	Suppose $x \sim y$ for some $ y \in U_{6n}{\setminus} P_i$. Then $x, y \in \langle z \rangle$ for some $z \in U_{6n}$. Note that $z \notin P_i$ and   so  $x \in U_{6n} {\setminus} P_i$, a contradiction of Proposition \ref{Common-elementP_i-Q_i} (i). Since $P_i$ is a cyclic subgroup and $x, y \in P_i$, the converse holds.
\end{proof}

The proof of the following lemma is also similar.

\begin{lemma}\label{degree-element-Q_i}
	Let $x \in Q_i {\setminus} \langle a \rangle$, where $0 \leq i \leq k +1$. Then $x \sim y$ for any $y \in U_{6n}$ if and only if $y \in Q_i$.
\end{lemma}

The following proposition determines neighbourhoods of vertices of $\mathcal{P}_e(U_{6n})$.

\begin{proposition}\label{nbd}
	For the graph $\mathcal{P}_e(U_{6n})$, we have
	\begin{enumerate}[\rm(i)]
		\item $N[x] = P_i$ if and only if $x \in P_i {\setminus} \langle a \rangle$, where $0 \leq i \leq k+1$.
		\item $N[x] = Q_i$ if and only if $x \in Q_i {\setminus} \langle a \rangle$, where $0 \leq i \leq k+1$.
		\item $N[x] = \langle a \rangle$ if and only if $x = a^i$ for some odd $i$.
		\item $N[x] = U_{6n}$ if and only if $x \in P_k \cap \langle a\rangle = \langle a^{2\cdot 3^k}\rangle$.
		\item $N[x] = \bigcup\limits_{j = 0}^{i -1}\left( P_j \cup Q_j\right) \cup P_{k +1} \cup Q_{k +1}\cup \langle a \rangle$ if and only if  $x \in \langle a^{2\cdot 3^i} \rangle {\setminus} \langle a^{2\cdot 3^{i + 1}} \rangle$, where $0 \leq i \leq k-1$.
	\end{enumerate}
\end{proposition}
\begin{proof}
	
	(i) If $x \in P_i {\setminus} \langle a \rangle$, then by Corollary \ref{degree-element-P_i}, we have $N[x] = P_i$. Conversely, suppose that $x \notin P_i {\setminus} \langle a \rangle$. Then either $x \in P_i \cap \langle a \rangle$ or $x \in U_{6n} {\setminus} P_i$. If  $x \in P_i \cap \langle a \rangle$, then clearly $x \sim a$. But $a \notin P_i$, so that $N[x]  \neq P_i$, a contradiction. On the other hand, if $ x \in U_{6n} {\setminus} P_i$, then by Corollary \ref{degree-element-P_i}, we have $N[x] \neq P_i$, again a contradiction. This proves (i).
	
	\smallskip
	\noindent
	(ii) The proof  is similar to that of (i).
	
	\smallskip
	\noindent
	(iii) Let  $x = a^i$ for some odd $i$. By \eqref{u6neq}, \eqref{u6neq1} and \eqref{u6neq2}, we have $a^i \notin P_j \cup Q_j$ for all $0 \leq j \leq k+1$.  Consequently, $a^i \sim x$ if and only if $x \in \langle a \rangle$. Thus $N[x] = \langle a \rangle$.
	
	Now let $x \in U_{6n}$ be such that $N[x] = \langle a \rangle$. If $x = a^{2i}$ for some $i$, then by \eqref{u6neq}, $x \in \langle ab \rangle$. This implies $x \sim ab$, which is a contradiction. Moreover, from  (i) and  (ii), we have $x \notin (P_i \cup Q_i) {\setminus} \langle a \rangle$. Thus $x = a^i$ for some odd $i$.
	
	\smallskip
	\noindent
	(iv) Let  $x \in P_k \cap \langle a \rangle =  \langle a^{2\cdot 3^k} \rangle$. By  Remark \ref{a^{23^i+1}-a^{23^i}}, $\langle a^{2\cdot 3^k} \rangle \subseteq \langle a^{2\cdot 3^{i}} \rangle$ for all $0 \leq i \leq k-1$.  By Lemma \ref{P_i-intersection-generated-a}, since $P_i \cap \langle a  \rangle = \langle a^{2\cdot 3^{i +1}} \rangle = Q_i \cap \langle a \rangle$, we have $a^{2\cdot 3^k} \in P_i$ $  \cap$ $  Q_i$ for all $ 0 \leq i \leq k$.  Accordingly, $x\in P_i \cap Q_i$ for any $0 \leq i \leq k$. Since $P_{k +1} \cap \langle a \rangle =   \langle a^2 \rangle =  Q_{k +1} \cap \langle a \rangle $, we get $x \in P_{k +1} \cap Q_{k+1}$. Thus $x \in \langle a \rangle\cap \left( P_i \cap Q_i\right)$ for any $i$, $0 \leq i \leq k+1$. For any  $y \in U_{6n}$,  by \eqref{u6neq2}, we have $x \sim y$.  Consequently, $x$ is a dominating vertex of $\mathcal{P}_e(U_{6n})$.
	
	Conversely, suppose $x$ is a dominating vertex. Then $x$ is adjacent with every element of $P_i$ and  $ Q_i$ for all $0 \leq i \leq k+1$. Consequently, $x \in P_i \cap Q_i \cap \langle a \rangle = \langle a^{2\cdot 3^{i +1}} \rangle$ for all $0 \leq i \leq k$. Hence  $x \in P_k \cap \langle a \rangle =  \langle a^{2\cdot 3^k} \rangle$.

	\smallskip
	\noindent
	(v) Let $x \in \langle a^{2\cdot 3^i} \rangle {\setminus} \langle a^{2\cdot 3^{i + 1}} \rangle$. Clearly, $i < k$.
	We prove that $$x \sim y \; \text{ if and only if} \; y \in \bigcup\limits_{j = 0}^{i -1}\left( P_j \cup Q_j\right) \cup P_{k +1} \cup Q_{k +1}\cup \langle a \rangle.  $$
	
	In order to prove this, by \eqref{u6neq}, we have $x \in P_{k +1} \cap Q_{k+1}$. Clearly, $x \in \langle a \rangle$. By Lemma \ref{P_i-intersection-generated-a}, we have  $P_{j} \cap \langle a \rangle =\langle  a^{2\cdot3^{j + 1}} \rangle$. Moreover, by  Remark \ref{a^{23^i+1}-a^{23^i}}, $x \in \langle a^{2\cdot 3^{j + 1}} \rangle$ for all $0 \leq j \leq i -1$. Thus it follows that $x \in P_j$. Similarly, one can observe $x \in Q_j$. Consequently,  for any $y \in \bigcup\limits_{j = 0}^{i -1}\left( P_j \cup Q_j\right) \cup P_{k +1} \cup Q_{k +1}\cup \langle a \rangle$, we have $y \sim x$. If possible, let $y \notin  \bigcup\limits_{j = 0}^{i -1}\left( P_j \cup Q_j\right) \cup P_{k +1} \cup Q_{k +1} \cup \langle a \rangle $.  Then in view of \eqref{u6neq2}, we have either $y \in P_j {\setminus} \langle a \rangle$ or $y \in Q_j {\setminus} \langle a \rangle$ for some $j$, where $i \leq j \leq k$. If $ x\sim y$, then by Corollary \ref{degree-element-P_i} and Corollary \ref{degree-element-Q_i}, we have either $x \in P_j$ or $x \in Q_j$. If $x \in P_j$, then clearly $ x \in P_j \cap \langle a \rangle = \langle  a^{2\cdot3^{j + 1}} \rangle$. Since $j \geq i$, by Remark \ref{a^{23^i+1}-a^{23^i}}, we get $x \in \langle  a^{2\cdot3^{i + 1}} \rangle$, a contradiction.  Similarly, we get a contradiction when $ x\in Q_j$. Consequently, $x \nsim y$. Hence if $x \sim y$, then $y \in \bigcup\limits_{j = 0}^{i -1}\left( P_j \cup Q_j\right) \cup P_{k +1} \cup Q_{k +1}\cup \langle a \rangle$.
\end{proof}
Applying preceding results of this subsection, we now show that the enhanced power graph of $U_{6n}$ is prefect.

\begin{theorem}
	The enhanced power graph  $\mathcal{P}_e(U_{6n})$ is perfect.
\end{theorem}

\begin{proof}
	In view of Theorem \ref{strongperfecttheorem}, it is enough to show that $\mathcal{P}_e(U_{6n})$ does not contain a hole or antihole of odd length greater than or equal to five. First suppose that $\mathcal{P}_e(U_{6n})$ contains a hole $C$ given by $x_1 \sim x_2\sim \dots \sim  x_{l} \sim x_1$, where $l \geq 5$. Then we have the following two cases and each of them ends up contradicting the fact that $C$ is a hole.
	
	\smallskip
	\noindent
	{\bf Case 1.}  $x_i \notin \langle a \rangle$ for all $i$. In view of  \eqref{u6neq2}, we obtain  either  $x_1 \in P_j {\setminus} \langle a \rangle$  or $x_1 \in Q_j {\setminus} \langle a \rangle$ for some $j$. Without loss of generality, we suppose that $ x_1 \in P_j {\setminus} \langle a \rangle$ for some $j$.  Since $x_{l}  \sim x_1 \sim x_2$, by Lemma \ref{degree-element-P_i}, we have  $x_{l}, x_2 \in P_j$. Consequently, $x_{l} \sim x_2$.
	
	\smallskip
	\noindent
	{\bf Case 2.}   $x_i \in \langle a \rangle$ for some $i$. Without loss of generality, we assume that $x_1 \in \langle a \rangle$. Since $x_4$ is not adjacent with $x_1$ so by \eqref{u6neq2}, either  $x_4 \in P_j {\setminus} \langle a \rangle$ or $x_4 \in Q_j {\setminus} \langle a \rangle$ for some $j$. Also, $x_3 \sim x_4 \sim x_5$, by Lemma \ref{degree-element-P_i} and Lemma \ref{degree-element-Q_i}, we have  either  $x_3, x_5 \in P_j$ or $x_3, x_5 \in Q_j$. We obtain $x_3 \sim x_5$.
	
	Now suppose $C'$ is an antihole of length at least $5$ in $\mathcal{P}_e(U_{6n})$, that is, we have a hole $y_1 \sim y_2 \sim \dots \sim y_{l} \sim y_1$, where $l \geq 5$, in $\overline{\mathcal{P}_e(U_{6n})}$. Then again we arrive at contradiction in each of the following cases.
	
	\smallskip
	\noindent
	{\bf Case 1.}  $y_i \notin \langle a \rangle$ for all $i$. Since $y_1 \notin \langle a \rangle$, by \eqref{u6neq2}, either  $y_1 \in P_j {\setminus} \langle a \rangle$  or $y_1 \in Q_j {\setminus} \langle a \rangle$ for some $j$.  Since $y_1 \sim y_3$ and $ y_1 \sim y_4 $ in $\mathcal{P}_e(U_{6n})$. By Lemma \ref{degree-element-P_i} and Lemma \ref{degree-element-Q_i}, we get either $y_3, y_4 \in P_j$ or $y_3, y_4 \in Q_j$. Thus we have  $y_3 \sim y_4$ in $\mathcal{P}_e(U_{6n})$.
	
	\smallskip
	\noindent
	{\bf Case 2.}  $y_i \in \langle a \rangle$ for some $i$. Without loss of generality, we assume that $y_1 \in \langle a \rangle$. Notice that we have $y_1 \nsim y_2$ in $\mathcal{P}_e(U_{6n})$. Consequently, either $y_2 \in P_j {\setminus} \langle a \rangle$  or $y_2 \in Q_j {\setminus} \langle a \rangle$. Moreover, $y_2 \sim y_4$ and  $y_2 \sim y_5$ in $\mathcal{P}_e(U_{6n})$, as $C'$ is an antihole in $\mathcal{P}_e(U_{6n})$. Thus either  $y_4, \;  y_5  \in P_j$ or $ y_4,  y_5 \in Q_j$. As a result,  $x _4 \sim y_5$ in $\mathcal{P}_e(U_{6n})$.
\end{proof}

In the following theorem, we compute various graph invariants under consideration for $\mathcal{P}_e(U_{6n})$.

\begin{theorem} For $n\geq 1$, the following hold:
	\begin{enumerate}[\rm(i)]
		\item The minimum degree of  $\mathcal{P}_e(U_{6n})$ is
		$$\delta(\mathcal{P}_e(U_{6n})) = \left\{ \begin{array}{ll}
		2t -1 & \mbox{if $k = 0$},\\
		3t -1 & \mbox{if $k > 0$.}\end{array}
		\right.$$
		\item The independence number of $\mathcal{P}_e(U_{6n})$ is  $2k + 4$.
		\item For $n > 1$, the matching number   $\alpha'(\mathcal{P}_e(U_{6n}))$  is $3n$  and $\alpha'(\mathcal{P}_e(U_{6})) = 2$.
		\item The strong metric dimension of $\mathcal{P}_e(U_{6n})$ is $ 6n - k - 2$.
	\end{enumerate}
\end{theorem}

\begin{proof}
		\noindent
		(i) By Proposition \ref{nbd}, we have
			\begin{enumerate}[\rm(a)]
				\item   $\deg(x) = 2n - 1$ for all $x \in (P_{k + 1} \cup Q_{k + 1}) {\setminus} \langle a \rangle$.
				\item $\deg(x) =3^{k -i}t - 1$   for all $x \in (P_{i} \cup Q_{i}) {\setminus} \langle a \rangle$, where $0 \leq i \leq k-1$.
				\item $\deg(x) =3t - 1$ for all $x \in P_k {\setminus} \langle a \rangle$.
				\item $\deg(a^{i}) = 2n - 1$,  where $i$ is {\rm odd}.
				\item $\deg(x) = 6n -1$,  where $x \in P_k \cap  \langle a\rangle$.
			\end{enumerate}
		
		Moreover, for $0 \leq i \leq k-1$ and $x \in \langle a^{2\cdot 3^i} \rangle {\setminus} \langle a^{2\cdot 3^{i + 1}} \rangle$, we get
		$$ N[x] = \bigcup\limits_{j = 0}^{i -1}\left( P_j \cup Q_j\right) \cup P_{k +1} \cup Q_{k +1}\cup \langle a \rangle.$$ 	
		As a result,
				\begin{align*}
				 \deg(x)& =  |\langle a \rangle|+ \displaystyle{\sum_{j = 0}^{ i-1}}(|P_j {\setminus} \langle a\rangle| + |Q_j{\setminus} \langle a\rangle|) + |P_{k +1} {\setminus} \langle a\rangle| + |Q_{k +1}{\setminus} \langle a\rangle| - 1\\
				& =  4n + \displaystyle{\sum_{j = 0}^{ i-1}2|P_j {\setminus} \langle a \rangle|} -1\\
				& =  4n + 2t3^{k -i}(3^{i} - 1) - 1.
				\end{align*}
		
		Let $x \in U_{6n}$. If $k = 0$, then clearly $n = t$. Then in view of \eqref{u6neq2} and from above,  $\deg(x) \in \{2n -1, 3n - 1, 6n -1\}$. Thus $\delta(\mathcal{P}_e(U_{6n})) = 2n -1 = 2t -1$.
		
		 Now take $k > 0$. From above we observe that $\deg(x) \geq 3t - 1$, and when $x \in P_k {\setminus} \langle a \rangle$, we have $\deg(x) =3t - 1$. Hence $\delta(\mathcal{P}_e(U_{6n})) = 3t -1$.
		
		\smallskip
		\noindent
		(ii) Consider the set $I = \{a, ab, ab^2,      a^{2\cdot 3^k}b  \} \cup \{ a^{2\cdot 3^i}b :   0 \leq i \leq k- 1  \} \cup \{a^{2\cdot 3^i}b^2  :   0 \leq i \leq k- 1  \}$. Then by Lemma \ref{degree-element-P_i} and Lemma \ref{degree-element-Q_i},  $I$ is an independent set of size $2k + 4$. If there exists another independent set $I'$ such that  $|I'| > 2k + 4$, then there exist $x, y \in I'$ with the following possibilities: (a) $x, y \in P_i$ for some $0 \leq i \leq k+1$, or (b) $x, y \in Q_j$ for some $0 \leq j \leq k+1$, or (c) $x, y \in \langle a \rangle$. For each case, we have  $x \sim y$, which is a contradiction. So that $\alpha(\mathcal{P}_e(U_{6n})) = 2k +4$.
		
		\smallskip
		\noindent
		(iii) The result is straightforward for $n = 1$. Now let $n > 1$. In order to prove that $\alpha'(\mathcal{P}_e(U_{6n}))$  is $3n$, we provide a partition of $V(\mathcal{P}_e(U_{6n}))$ into subsets of even size such that the subgraph induced by each subset is complete. Note that \eqref{u6neq2} can be written as
		$$U_{6n} = \langle a \rangle \cup \bigcup_{i = 0}^{k-1}\left(P_i {\setminus} \langle a \rangle \cup Q_i {\setminus} \langle a \rangle \right)\cup \left(P_k {\setminus} \langle a \rangle\right) \cup \left(P_{k+1} {\setminus} \langle a \rangle\right) \cup \left(Q_{k+1} {\setminus} \langle a \rangle\right).$$
		
		These sets on the the right hand side of above expression forms a partition of $V(\mathcal{P}_e(U_{6n}))$. For $0 \leq i \leq k$, the subsets $P_i  {\setminus} \langle a \rangle$ and $Q_i {\setminus} \langle a \rangle$ are of even cardinality (cf. Lemma \ref{order-element-U_6n} (i)). If $n$ is even, then by Lemma \ref{order-element-U_6n}, the subsets  $P_{k+1} {\setminus} \langle a \rangle$ and  $Q_{k+1} {\setminus} \langle a \rangle$ are  of even size. Also, the subgraph induced by each subset of the given partition is complete. Consequently, $\alpha'(\mathcal{P}_e(U_{6n})) = 3n$.
		
		Also, notice that the subsets on the right hand side of the following expression
		\begin{align*}
			 U_{6n} = & \left( \langle a \rangle {\setminus} \{e, a^2\} \right) \cup \bigcup\limits_{i = 0}^{k-1}\left(P_i {\setminus} \langle a \rangle \cup Q_i {\setminus} \langle a \rangle \right)\\
			& \cup \left( P_k {\setminus} \langle a \rangle\right) \cup \left(\left( P_{k+1}{\setminus} \langle a \rangle\right)  \cup \{e\}\right)  \cup \left( \left(Q_{k+1} {\setminus} \langle a \rangle \right)\cup \{a^2\}\right)
		\end{align*}
		forms a partition of $V(\mathcal{P}_e(U_{6n}))$. If $n$ is odd, then by Lemma \ref{order-element-U_6n}, size of each subset of this partition is even. Thus we have $\alpha'(\mathcal{P}_e(U_{6n})) = 3n$.
		
		\smallskip
		\noindent
		(iv)		We denote
		$$ V_1 = \{ \widehat{ab}, \widehat{ab^2}\} \cup \{  \widehat{a^{2\cdot 3^i}b}  :  0 \leq i \leq k \} \cup \{ \widehat{a^{2\cdot3^i}b^2}  :  0 \leq i \leq k-1\} $$
		and
		 $$ V_2 = \{\widehat{e}, \widehat{a} \} \cup \{ \widehat{a^{2\cdot 3^j}}  : 0 \leq j \leq k-1 \}.$$
		 	Then in view of the properties of $U_{6n}$ derived earlier and Proposition \ref{nbd}, $\widehat{U_{6n}} = V_1 \cup V_2$.
		  For any $\widehat{x} \in V_2$, we have $x  \in \langle a \rangle$. So that $V_2$  is a clique in $\widehat{\mathcal{P}}_e (U_{6n})$, and thus  $\omega(\widehat{\mathcal{P}}_e (\mathcal{R}_{U_{6n}})) \geq  k + 2$.
		  If possible, let $C$ be another clique in $\widehat{\mathcal{P}}_e ({U_{6n}})$ with $|C| > k +2$.

		   By Lemma \ref{degree-element-P_i} and Lemma \ref{degree-element-Q_i}, $x \not\sim y$ for every pair of distinct elements $\widehat{x}$, $\widehat{y}$ in $V_1$. Accordingly, $V_1$ is an independent set in $\widehat{\mathcal{P}}_e ({U_{6n}})$. Then $|V_1 \cap C| \leq 1$, so that $V_2 \subset C$. In fact, comparing the cardinalities, $|V_1 \cap C| = 1$. We denote $V_1 \cap C  = \{\widehat{x}\}$.
		   Consequently,  $x \in P_i {\setminus} \langle a \rangle$ or $x \in Q_i {\setminus} \langle a \rangle$ for some $0 \leq i \leq k +1$.
		    Moreover, $\widehat{a} \in C$, so that $x \sim a$ in $\mathcal{P}_e (U_{6n})$.
		   Then by Lemma \ref{degree-element-P_i} and Lemma \ref{degree-element-Q_i}, $a \in P_i$ or $a \in Q_i$. Since this is a contradiction, we get $\omega(\widehat{\mathcal{P}}_e (U_{6n})) = k + 2$.  Hence $\operatorname{sdim}(\mathcal{P}_e(U_{6n})) = 6n - k - 2$,  by Theorem \ref{strong-metric-dim}.
\end{proof}

\bigskip
\subsection{Dihedral group.}\hfill
\smallskip

 For $n \geq 3$, the \emph{dihederal group} $D_{2n}$ of order $2n$ is given by the presentation $$D_{2n} = \langle a, b \; : \; a^{n} = b^2 = e, \; ab = ba^{-1} \rangle.$$

It is known that every element of $D_{2n} {\setminus} \langle a \rangle$ is of the form $a^ib$ for some $0 \leq i \leq n-1$, and that $ \langle a^ib \rangle =   \{ e, a^ib \}  $. In particular, $$D_{2n} = \langle a \rangle \cup \bigcup\limits_{ i=  0}^{n-1} \langle a^ib \rangle.$$

\begin{theorem}
	The enhanced power graph of $D_{2n}$ is perfect.
\end{theorem}

\begin{proof}
	We apply Theorem \ref{strongperfecttheorem} to prove the theorem.
	 Let $C$ be a hole of $\mathcal{P}_e(D_{2n})$.  We have $\deg(a^ib) = 1$, so that $a^ib \notin C$ for all $0 \leq i \leq n-1$. Thus the vertices of $C$ belong to $\langle a \rangle$. Since the subgraph induced by $\langle a \rangle$ is complete, $C$ is a cycle of length three.
	
	 Now let $C'$ be an antihole of length at least $5$ of $\mathcal{P}_e(D_{2n})$. If possible, suppose that $V(C') \cap \langle a \rangle \neq \emptyset$. Then there exists $x_1 \in C' \cap \langle a \rangle$ such that $x_1 \sim x_2$ in $\overline{\mathcal{P}_e(D_{2n})}$ for some vertex $x_2$ of $C'$. Equivalently, $x_1 \nsim x_2$ in $\mathcal{P}_e(D_{2n})$.  Thus $x_2 = a^ib $ for some $i$. As $|V(C')| \geq 5$, there exists $x_3 \in V(C') {\setminus} \{x_1, x_2\}$ such that $ x_1 \sim x_3$ and $x_2 \nsim x_3$ in $\mathcal{P}_e(D_{2n})$. This implies $x_3 \in \langle a \rangle$. Similarly, there exists $x_4 \in V(C') {\setminus} \{x_1, x_2, x_3\}$ such that $x_3 \nsim x_4,$ $x_1 \sim x_4$ and  $x_2 \sim x_4$ in $\mathcal{P}_e(D_{2n})$.  Note that none of these $x_i's$ are $e$. As $x_1 \sim x_4$, we get $x_1 \in \langle a \rangle$, whereas, $x_3 \nsim x_4$ yield $x_1 \not\in \langle a \rangle$. Since this is impossible, $V(C') \cap \langle a \rangle = \emptyset$. That is, every element of $V(C')$ is of the form $a^ib$. However, the subgraph of $\overline{\mathcal{P}_e(D_{2n})}$ induced by the set $\{a^ib : 0 \leq i \leq n-1\}$ is complete. This contradicts the fact that the length of $C'$ is at least $5$.
	
	 Consequently, the proof follows from Theorem \ref{strongperfecttheorem}.
\end{proof}

\begin{theorem} For $n\geq 2$, we have the following results:
	\begin{enumerate}[\rm(i)]
	    \item The minimum degree  of $\mathcal{P}_e(D_{2n})$ is  $1$.
		\item The independence number of $\mathcal{P}_e(D_{2n})$ is $n+1$.
		\item   The matching number of $\mathcal{P}_e(D_{2n})$ is $\lceil \frac{n}{2} \rceil$.
		\item The strong metric dimension of  $\mathcal{P}_e(D_{2n})$ is  $2(n -1)$.
	\end{enumerate}
\end{theorem}

\begin{proof}
(i)	Since the only vertex adjacent to $ a^ib $ is $e$ for any $0 \leq i \leq n-1$, the proof follows.
	
\smallskip

\noindent
(ii) Observe that the set $I = \{a\} \cup  \{a^ib : 0 \leq i \leq n-1\}$  is an independent set, and thus $\alpha (\mathcal{P}_e(D_{2n})) \geq n + 1$. If there exists an independent set $I'$ such that $|I'| > n +1$, then we must have $x, y \in I'$ such that $x, y  \in \langle a \rangle$. However, this results in $x \sim y$, which is a contradiction. As a result, $\alpha(\mathcal{P}_e(D_{2n})) = n+1$.
		
\smallskip
 \noindent
		(iii) If $n$ is even, then observe that the size of maximum matching is $\frac{n}{2}$ which can be constructed from the complete graph induced by  $\langle a \rangle$. If $n$ is odd, then the size of maximum matching is $\lceil \frac{n}{2} \rceil$ which can be constructed $\frac{n-1}{2}$ edges of  $\langle a \rangle {\setminus} \{e\}$ and one edge $(ab, e)$ of $H_1$. Therefore, $\alpha'(\mathcal{P}_e(D_{2n})) = \lceil \frac{n}{2} \rceil$.
	
\smallskip \noindent
(iv)  The $\equiv$-classes in $D_{2n}$ are $\widehat{ e}$,
$\widehat{ a}$, $ \widehat{ b}$, $ \widehat{ ab}$, \dots, $ \widehat{ a^{n-1}b}$, where $\widehat{ e} =\{e\}$, $\widehat{ a} = \langle a \rangle {\setminus} \{e\}$ and $\widehat{ a^ib} = \{a^ib\}$. Then in view of the adjacency relation of elements of these classes in $\mathcal{P}_e(D_{2n})$, we have $\widehat{\mathcal{P}}_e(\mathcal{R}_{D_{2n}}) \cong K_{1, n+1} $. So that  $\omega(\widehat{\mathcal{P}}_e(\mathcal{R}_{D_{2n}})) = 2$. Hence $\operatorname{sdim}(\mathcal{P}_e(D_{2n})) = 2(n -1)$, by Theorem \ref{strong-metric-dim}.
\end{proof}

\bigskip
\subsection{Semidihedral group.}\hfill
\smallskip

 For $n \geq 2$, the \emph{semidihedral group} $SD_{8n}$ is a group of order $8n$ with presentation  $$SD_{8n} = \langle a, b  :  a^{4n} = b^2 = e,  ba = a^{2n -1}b \rangle.$$
 We have $$ ba^i = \left\{ \begin{array}{ll}
a^{4n -i}b & \mbox{if $i$ is even,}\\
a^{2n - i}b& \mbox{if $i$ is odd,}\end{array} \right.$$
so that every element of $SD_{8n} {\setminus} \langle a \rangle$ is of the form $a^ib$ for some $0 \leq i \leq 4n-1$. We denote the subgroups $H_i = \langle a^{2i}b \rangle = \{e, a^{2i}b\}$ and $ T_j =  \langle a^{2j + 1}b \rangle = \{e, a^{2n}, a^{2j +1}b, a^{2n + 2j +1}b\} $. Then we have $$SD_{8n} = \langle a \rangle \cup \left( \bigcup\limits_{ i=0}^{2n-1} H_i \right) \cup \left( \bigcup\limits_{ j=  0}^{n-1} T_{j}\right).$$

\begin{theorem}
	The enhanced power graph of $SD_{8n}$ is perfect.
\end{theorem}

\begin{proof}
	We utilize the notion of hole or antihole once  again to prove the theorem. First suppose $C$ is a hole of $\mathcal{P}_e(SD_{8n})$. For any $0 \leq i \leq 2n-1$, we notice $\deg(a^{2i}b) = 1$, so that $a^{2i}b \notin V(C)$. Since $a^{2n} \sim x$ for all $x \in \langle a \rangle \cup T_j$, we have $a^{2n} \notin V(C)$. Then we observe that N$[x] = \langle a \rangle$ if and only if $x \in \langle a \rangle {\setminus} \{e, a^{2n}\}$, and that N$[x] = T_j$ if and only if $x \in T_j {\setminus} \{e, a^{2n}\}$. Additionally, $e \notin V(C)$ as well. Thus all vertices of $C$ either belong to $ \langle a \rangle {\setminus} \{e, a^{2n}\}$ or $ T_j {\setminus} \{e, a^{2n}\}$. As $|T_j {\setminus} \{e, a^{2n}\}| = 2$, we have $V(C) \not\subset T_j {\setminus} \{e, a^{2n}\}$. Accordingly, $V(C) \subseteq \langle a \rangle {\setminus} \{e, a^{2n}\}$. Hence the length of $C$ is $3$, as $\langle a \rangle$ induces a complete subgraph in $\mathcal{P}_e(SD_{8n})$.
	
	Next, if possible, let $C'$ be an antihole of $\mathcal{P}_e(SD_{8n})$ of length atleast five. Since $a^{2i}b$ is adjacent with every element except $e$ in  $\overline{\mathcal{P}_e(SD_{8n})}$, and $e \notin V(C')$, we have $a^{2i}b \notin V(C')$. Now suppose $V(C') \cap \langle a \rangle \neq \emptyset$. Then there exists $x_1 \in V(C') \cap \langle a \rangle$ such that $x_1 \sim x_2$ in $\overline{\mathcal{P}_e(SD_{8n})}$ for some $x_2 \in V(C')$. Then $x_1 \nsim x_2$ in $\mathcal{P}_e(SD_{8n})$, so that $x_2 \in T_j {\setminus} \{e, a^{2n}\}$ for some $j$. Since $|V(C')| \geq 5$, there exists $x_3 \in V(C') {\setminus} \{x_1, x_2\}$ such that $ x_1 \sim x_3$ and $x_2 \nsim x_3$ in $\mathcal{P}_e(SD_{8n})$. As a result, $x_3 \in \langle a \rangle {\setminus} \{e, a^{2n}\}$. Furthermore, there exists $x_4 \in V(C') {\setminus} \{x_1, x_2, x_3\}$ such that $x_3 \nsim x_4,$ $x_1 \sim x_4$ and  $x_2 \sim x_4$ in $\mathcal{P}_e(SD_{8n})$. Then $x_1 \sim x_4$ and $x_3 \nsim x_4$ imply, respectively, that $x_3 \in \langle a \rangle$ and $x_3 \not\in \langle a \rangle$. Since this is impossible, $V(C') \cap \langle a \rangle = \emptyset$. Consequently, every vertex of $C'$ is of the form $a^{2i +1}b$. However, for any $0 \leq i \leq 2n-1$, $a^{2i +1}b$  is adjacent to every vertex in $\{a^{2j +1}b : 0 \leq j \leq 2n-1\} {\setminus} \{a^{2i +1}b, a^{2n+2i +1}b\}$ in $\overline{\mathcal{P}_e(SD_{8n})}$. This contradicts our assumption that  $C'$ is an antihole of $\mathcal{P}_e(SD_{8n})$ of length atleast five.
	
	Following Theorem \ref{strongperfecttheorem}, we therefore conclude that $\mathcal{P}_e(SD_{8n})$ is perfect.
\end{proof}

 Now we investigate graph invariants of $\mathcal{P}_e(SD_{8n})$ in the  following theorem.

\begin{theorem} For $n\geq 1$, we have the following results:
\begin{enumerate}[\rm(i)]
\item The minimum degree of  $\mathcal{P}_e(SD_{8n})$ is   $1$.
\item The independence number of  $\mathcal{P}_e(SD_{8n})$ is $3n + 1 $
\item   The matching number of  $\mathcal{P}_e(SD_{8n})$ is $3n$.
\item The strong metric dimension of $\mathcal{P}_e(SD_{8n})$ is  $8n -3$.
\end{enumerate}
\end{theorem}

\begin{proof}
(i) The proof follows from the fact that $e$ is the only vertex adjacent to $a^{2i}b$ for any $0 \leq i \leq 2n-1$.

 \smallskip

 \noindent
 (ii) Note that the set $I = \{a \} \cup \{a^{2i}b : 0 \leq i \leq 2n-1 \}  \cup \{a^{2j + 1}b : 0 \leq j \leq n-1 \}$ is an independent set in $\mathcal{P}_e(SD_{8n})$ so $\alpha (\mathcal{P}_e(SD_{8n})) \geq 3n + 1$. If possible, suppose there exists an independent set $I'$ such that $|I'| > 3n +1$. Then there exist $x, y \in I'$ such that  $x, y  \in \langle a \rangle$, $x, y \in H_i$ for some $i$ or $x, y \in T_j$ for some $ j$. Since  subgraphs induced by  $\langle a \rangle, \; H_i$ and $ T_j$, respectively  forms  a clique, we have $x \sim y$ for each of the possibility, a contradiction. Accordingly,  $\alpha(\mathcal{P}_e(SD_{8n})) = 3n+1$.

 \smallskip

 \noindent
(iii) Let $M$ be an matching in $\mathcal{P}_e(SD_{8n})$. Consider the set $U$ of endpoints of edges in $M$. We observe that  $a^{2i}b \sim x$ if and only if $x = e$. As a result,  $|U| \leq 6n +1$. However, as $M$ is a matching, $|U|$ is even. Then $|U| \leq 6n$ and thus $|M| \leq 3n$. Now let $\epsilon_i$ be the edge with endpoints $a^i, a^{2n +i}$, and $\epsilon'_j$ be the edge with endpoints $a^{2j + 1}b, a^{2n + 2j + 1}b$. Then the set  $M' = \{\epsilon_i  :  0 \leq i \leq 2n-1 \} \cup \{\epsilon'_j  :  0 \leq j \leq n-1 \}$ is a matching of size $3n$ in $\mathcal{P}_e(SD_{8n})$. Hence we get $\alpha'(\mathcal{P}_e(SD_{8n})) = 3n$.

 \smallskip

 \noindent

(iv) From the structure of $SD_{8n}$, we have $\widehat{SD_{8n}} = \{\widehat{ e}, \widehat{a^{2n}}, \widehat{a}\} \cup \{\widehat{a^{2i}b} : 0 \leq i \leq 2n-1\} \cup \{\widehat{a^{2j + 1}b} : 0 \leq j \leq n-1\}$, where $\widehat{e} =\{e\}$, $\widehat{a^{2n}} = \{a^{2n}\}$, $\widehat{a} = \langle a \rangle {\setminus}\{e, a^{2n} \},  \widehat{a^{2i}b} = \{a^{2i}b\}$, and $\widehat{a^{2j + 1}b} = \{a^{2j + 1}b, a^{2n + 2j + 1}b\}$.
 Furthermore, \break ${\widehat{SD_{8n}}} {{\setminus}} \{\widehat{e}, \widehat{a^{2n}}\}$ is an independent set, and each of $\widehat{e}$ and $\widehat{a^{2n}}$ are adjacent to the rest of the vertices  in $\widehat{\mathcal{P}}_e(\mathcal{R}_{SD_{8n}})$.  Thus we get $\omega({\widehat{\mathcal{P}}_e(\mathcal{R}_{SD_{8n}}}) = 3$. Finally by Theorem \ref{strong-metric-dim},  $\operatorname{sdim}(\mathcal{P}_e(SD_{8n})) = 8n -3$.
\end{proof}

\section{Acknowledgement}
The third author wishes to acknowledge the support of MATRICS Grant \break (MTR/2018/000779) funded by SERB, India.

%
%
%

\break

\noindent
{\bf Ramesh Prasad Panda\textsuperscript{\normalfont 1}, Sandeep Dalal\textsuperscript{\normalfont 2}, Jitender Kumar\textsuperscript{\normalfont 2}}

\bigskip

\noindent{\bf Addresses}:

\vspace{5pt}

\noindent
\textsuperscript{\normalfont 1}School of Mathematical Sciences, National Institute of Science Education and Research Bhubaneswar, P.O. - Jatni, District - Khurda, Odisha - 752050, India.

\vspace{5pt}

\noindent	
\textsuperscript{\normalfont 1}Homi Bhabha National Institute, Training School Complex, Anushakti Nagar, Mumbai - 400094, India.

\vspace{5pt}

\noindent
\textsuperscript{\normalfont 2}Department of Mathematics, Birla Institute of Technology and Science Pilani, Pilani - 333031, India.

\bigskip

\noindent{\bf Email addresses}: {deepdalal10@gmail.com (S. Dalal), jitenderarora09@gmail.com (J. Kumar)}
\end{document}